\documentclass[reqno,11pt]{amsart}

\usepackage{xcolor}
\usepackage{graphicx}
\usepackage[hidelinks]{hyperref}

\usepackage{amscd}

\usepackage{amsmath}
\usepackage{amsthm}
\usepackage{amssymb}
\usepackage{latexsym,array}
\usepackage{amsfonts}
\usepackage{shadow}
\usepackage{amsbsy}
\usepackage{amssymb}
\usepackage{dsfont}
\usepackage{mathtools}
\usepackage{enumitem}
\usepackage{doi}

\usepackage{tikz-cd}
\usepackage{color}
\usepackage{graphicx}
\usepackage[hidelinks]{hyperref}
\usepackage{cleveref}
\usepackage{fullpage}
\usepackage{comment}

\usepackage{dsfont}

\usepackage{cleveref}

\usepackage{amscd}
\usepackage{tikz-cd}
\usepackage{tikz}

\newtheorem{theorem}{Theorem}[section]
\newtheorem{proposition}[theorem]{Proposition}
\newtheorem{lemma}[theorem]{Lemma}

\newtheorem{remark}[theorem]{Remark}
\newtheorem{definition}[theorem]{Definition}

\title[The resolvent equations for the Harmonic and bi-Harmonic functional calculi in dimension five]{The resolvent equations for the Harmonic and bi-Harmonic functional calculi in dimension  five}

\author[Fabrizio Colombo]{F. Colombo}
\address{(FC) Politecnico di Milano, Dipartimento di Matematica, Via E. Bonardi 9, 20133 Milano, Italy}
\email{fabrizio.colombo@polimi.it}

\author[Antonino De Martino]{A. De Martino}
\address{(ADM) Politecnico di Milano, Dipartimento di Matematica, Via E. Bonardi 9, 20133 Milano, Italy}
\email{antonino.demartino@polimi.it}

\author[Joao Marques Da Costa]{J. Marques Da Costa}
\address{(JM)  Politecnico di Milano, Dipartimento di Matematica, Via E. Bonardi 9, 20133 Milano, Italy}
\email{joao.marquesdacosta@polimi.it}

\begin{document}

	\begin{abstract}
		The fine structures on the $S$-spectrum constitute a new research area that includes a class of functional calculi based on the $S$-spectrum and on integral transforms determined by the Fueter--Sce mapping theorem and the Cauchy formula for slice hyperholomorphic functions.
		This strategy, based on integral transforms, allows us to construct functional calculi that include harmonic and polyharmonic functional calculi. The resolvent operators in this setting do not arise directly from a Cauchy kernel, but rather from suitable manipulations of it. For this reason the corresponding resolvent equations differ substantially from those associated with the classical Cauchy kernel. In this paper, we investigate the harmonic and biharmonic resolvent equations in dimension five, as well as the corresponding product rules and Riesz projectors for these functional calculi.
	\end{abstract}

	\maketitle
	
	AMS Classification 47A10, 47A60. \medskip
	
	 \medskip
	
	\textbf{Acknowledgements:} Fabrizio Colombo and Antonino De Martino are supported by MUR grant Dipartimento di Eccellenza 2023-2027.
	
	\section{Introduction}

	The spectral theory on the $S$-spectrum has nowadays several applications.
	For an introduction to this theory, we refer the reader to \cite{FJBOOK, CGK, ColomboSabadiniStruppa2011}.
	Recently, a new branch of this spectral theory, focusing on the fine structures of the $S$-spectrum, has emerged. Quaternionic function theory and the associated operator theory are extensively treated in \cite{CDPS, CDS, CPS1, AD, DPS, Polyf1, Polyf2}, while the Clifford setting is discussed in \cite{CDP25, CDP2026, Fivedim}. Many developments in this field are currently in progress.
	The theory of fine structures involves function spaces linked to the Fueter-Sce mapping theorem (also known as the Fueter-Sce extension theorem),
	see \cite{Fueter, TaoQian1, Sce} and \cite{ColSabStrupSce, DDG1}.  In order to explain better this theorem we briefly recall the notion of Clifford Algebra.
	\newline
	\newline
	We denote by $\mathbb{R}_n$ the real Clifford algebra over $n$ imaginary units $e_1, \dots, e_n$ satisfying the relations $e_\ell e_m + e_m e_\ell = 0, \ell \neq m, e_\ell^2 = -1$. An element in the Clifford algebra is denoted by $\sum_A e_A x_A$ where $A = \{\ell_1 \dots \ell_r\} \in \mathcal{P}\{1, 2, \dots, n\}$, $\ell_1 < \dots < \ell_r$ is a multi-index and $e_A = e_{\ell_1} e_{\ell_2} \dots e_{\ell_r}, e_\emptyset = 1$.
	
	A point $(x_0, x_1, \dots, x_n) \in \mathbb{R}^{n+1}$ is identified with the element $x = x_0 + \underline{x} = x_0 + \sum_{j=1}^n x_j e_j \in \mathbb{R}_n$ called paravector. The real part $x_0$ of $x$ is also be denoted by $\text{Re}(x)$, and the vector part by $\underline{x} = x_1 e_1 + \dots + x_n e_n$. The conjugate of $x$ is denoted by $\bar{x} = x_0 - \underline{x}$ and the modulus is given by $|x|^2 = x_0^2 + \dots + x_n^2$.
	We introduce the notation $\mathbb{S}$ for the sphere of unit 1-vectors in $\mathbb{R}^n$, defined as
	$
	\mathbb{S}=\{ \underline{x}=e_1x_1+\ldots +e_nx_n\ :\ x_1^2+\ldots +x_n^2=1\}.
	$
	\newline
	\newline
	In the setting of Clifford algebras, two classes of functions play a prominent role: slice hyperholomorphic functions and monogenic functions (see Definitions \ref{sh} and \ref{am}, respectively). The Fueter-Sce mapping theorem bridges slice hyperholomorphic functions $\mathcal{SH}(U)$, defined on a suitable set $U\subseteq \mathbb{R}^{n+1}$, and  axially monogenic functions, this is a subclass of monogenic function  and will be denoted by $\mathcal{AM}(U)$.
	
	This connection is established using the second operator associated with the Fueter-Sce mapping theorem given by the operator $\Delta^{\frac{n-1}{2}}_{n+1}$, where $\Delta_{n+1}$ is the Laplacian in $n+1$ dimensions (see Theorem \ref{FSTheo}), with $n$ odd.
	Recall that the Dirac operator $D$ and its conjugate $\overline{D}$ are defined as:
	\begin{equation}\label{DIRACeBARDNN}
		D = \frac{\partial}{\partial x_0} + \sum_{i=1}^{n} e_i \frac{\partial}{\partial x_i}, \quad \text{and} \quad
		\overline{D} = \frac{\partial}{\partial x_0} - \sum_{i=1}^{n} e_i \frac{\partial}{\partial x_i}.
	\end{equation}
	A consequence of the Fueter-Sce mapping theorem is that:
	\begin{equation}\label{FSconsequence}
		D\Delta^{\frac{n-1}{2}}_{n+1}f(x)=0, \quad \text{for all } f\in \mathcal{SH}(U).
	\end{equation}

	We define the \textit{Dirac fine structures of the spectral theory on the $S$-spectrum} as the collection of function spaces and their corresponding functional calculi arising from the various factorizations of the operator $\Delta^{\frac{n-1}{2}}_{n+1}$ in terms of the Dirac operator $D$ and its conjugate $\overline{D}$.
	
	The operator $\Delta^{\frac{n-1}{2}}_{n+1}$ admits multiple factorizations based on the identity $D\overline{D} = \Delta_{n+1}$. For instance, one may employ alternating products or grouped powers:
	\[
	\Delta^{\frac{n-1}{2}}_{n+1}
	= \underbrace{(D \overline{D}) \cdots (D \overline{D})}_{(n-1)/2-\text{times}}
	\qquad \text{or} \qquad
	\Delta_{n+1}^{\frac{n-1}{2}}
	= \underbrace{D \cdots D}_{(n-1)/2- \text{times}} \cdot
	\underbrace{\overline{D} \cdots \overline{D}}_{(n-1)/2-\text{times}}.
	\]
	Moreover, since $D$ and $\overline{D}$ commute with their composition $\Delta_{n+1}$ (and indeed pairwise when acting on relevant function spaces), any permutation of the operators is permissible. This leads to the general factorization:
	\[
	\Delta_{n+1}^{\frac{n-1}{2}} = \mathcal{P}\Bigg( \underbrace{D, \dots, D}_{k\text{-times}}, \underbrace{\overline{D}, \dots, \overline{D}}_{k\text{-times}} \Bigg),
	\]
	where $k = \frac{n-1}{2}$ and $\mathcal{P}$ represents any arbitrary ordering of the $2k$ operators.
	
	Consequently, the fundamental building blocks of any Dirac-factorization of $\Delta_{n+1}^{\frac{n-1}{2}}$ take the form
	\begin{equation}\label{oper-D-Delta}
		D^\beta \Delta_{n+1}^m \qquad \text{and} \qquad \overline{D}^\beta \Delta_{n+1}^m,
	\end{equation}
	for suitable integers $\beta, m \in \mathbb{N}_0$. This observation highlights classes of functions, defined via the operators in \eqref{oper-D-Delta}, that plays a central role in this theory.
	Let $U$ be an open set in $\mathbb{R}^{n+1}$ and let $\beta, m \geq 0$ be integers.
	A function $f: U \to \mathbb{R}_n$ of class $\mathcal{C}^{2m+\beta}(U)$ is called
	axially Analytic-Harmonic function of type $(\beta, m)$ if it is of axial type and satisfies:
	\begin{equation}
		D^\beta \Delta_{n+1}^m f(x)=0, \qquad \forall x \in U \quad \left(\text{resp. } f(x)\Delta_{n+1}^m D^\beta=0, \quad \forall x \in U \right).
	\end{equation}
	We denote this class of functions by $\mathcal{AAH}_{\beta,m}(U)$
	The particular and important subclass of type $(0, 1)$ contains the harmonic functions, often denoted by
	$\mathcal{AH}(U)$, and the subclass of type $(0, m)$ consists of
	polyharmonic function of order $m\geq2$ which is commonly indicated as $\mathcal{APH}_m(U)$.

	\medskip
	Using the explicit expressions for the kernels $D^\beta \Delta^m_{n+1} S^{-1}_{L}(s,x)$ and $\overline{D}^\beta \Delta^m_{n+1} S^{-1}_{L}(s,x)$ (cf. \cite{DP2026}), the Cauchy formula for slice hyperholomorphic functions and the Fueter-Sce mapping theorem, we obtain an integral representation for axially Analytic-Harmonic functions $\mathcal{AAH}_{\beta,m}(U)$ of type $(\beta,m)$, with $\beta, m \in \mathbb{N}_0$. Namely, applying the operator $D^\beta \Delta^m_{n+1}$ to a left slice hyperholomorphic function $f(x)$ yields:
	\begin{equation}\label{INTEGRALREP}
		D^\beta \Delta_{n+1}^m f(x) = \frac{1}{2 \pi}\int_{\partial (U\cap \mathbb{C}_I)} D^\beta \Delta_{n+1}^m S_L^{-1}(s,x)\, ds_I\, f(s).
	\end{equation}
	These integral representations allow us  to define the functional calculi corresponding
	to all functions arising from the factorizations of $\Delta^{\frac{n-1}{2}}_{n+1}$.
	We define
	$$
	f_{\beta,m}(x):=D^\beta \Delta_{n+1}^m f(x), \ \ \  \ \ S_{D^\beta \Delta^m, L}^{-1}(s,x):=D^\beta \Delta_{n+1}^m S_L^{-1}(s,x),\ \ \ \ {\rm with} \ \ \ \beta, m \in \mathbb{N}_0
	$$
	and the related resolvent operators are obtained by the replacement of $x$ by $T$ in $S_{D^\beta \Delta^m,L}^{-1}(s,x)$ where $T$
	is a  bounded paravector operators with commuting components.
	Assuming the $T$ has the $S$-spectrum contained in a suitable open set $U$,
	the associated functional calculus reads as:
	\begin{equation}
		\label{finsfun}
		f_{\beta,m}(T) := \frac{1}{2\pi} \int_{\partial(U \cap \mathbb{C}_I)}S_{D^\beta \Delta^m, L}^{-1}(s,T) \, ds_I \, f(s).
	\end{equation}
	Similar expression is obtained for the right resolvent operators $S_{D^\beta \Delta^m, R}^{-1}(s,T)$ and the related functional calculi.
	This framework unifies various functional calculi based on the $S$-spectrum $\sigma_S(T)$. While $\sigma_S(T)$ was initially introduced for slice hyperholomorphic functions, it serves as the foundation for a richer hierarchy of calculi, including the poly-harmonic, poly-analytic (generalizing \cite{FISH,Aro}) and monogenic functional calculi.
	
	The resolvent operators within the fine structures do not come directly from a standard Cauchy formula, but from applying differential operators to the Cauchy kernel, as in the $F$-resolvent operator \cite{CDS1}.

	\medskip
	The $S$-functional calculus, originally defined for Clifford operators with non-commuting components, admits a commutative version for paravector operators $T $ with commuting components. In this setting, the F-spectrum $\sigma_F(T)$
	which is defined via the invertibility of the operator:
	\[
	\mathcal{Q}_{c,s}(T) = s^2 \mathcal{I} - s(T + \overline{T}) + T\overline{T},
	\]
	coincides with the S-spectrum $\sigma_S(T)$,
	that is defined via the invertibility of the operator
	\[
	\mathcal{Q}_s(T) := T^2 - 2\text{Re}(s)T + |s|^2.
	\]
	This leads to the definition of the commutative left and right S-resolvent operators, $S_L^{-1}(s,T)$ and $S_R^{-1}(s,T)$, and the commutative pseudo S-resolvent operator $\mathcal{Q}_{c,s}^{-1}(T)$, see \eqref{eq:1.7} and \eqref{eq:1.9}, respectively.
	
	A crucial fact of this calculus is that the S-resolvent equation differs significantly from the classical holomorphic case, see \cite{ACGS}. The equation relates the product $S_R^{-1}(s,T)S_L^{-1}(p,T)$ to the difference of the resolvents, entangled with the slice hyperholomorphic Cauchy kernel (cf.  Theorem \ref{Bres}).
	Unlike standard resolvent equations, this relation requires both left and right resolvent operators
	to preserve the appropriate slice hyperholomorphicity in the variables $s$ and $p$.
	This structure ensures that the $S$-functional calculus maintains the necessary analyticity properties.
	\newline
	In \cite{Polyf2}, a resolvent equation and a product formula were derived for a harmonic functional calculus in the quaternionic case ($n = 3$), arising from the factorization of the operator $\Delta_4$.
	
	Our main results focus on the resolvent equations for the harmonic and biharmonic functional calculi arising from the factorizations of the second map in the Fueter-Sce mapping theorem with $n=5$, i.e., of $\Delta_6^2$. Concretely, we derive the resolvent equation for the functional calculi in \eqref{finsfun} by considering the cases $\beta = 1$, $m = 0$, and $\beta = m = 1$. In these instances, the resolvent operators are given by
	\begin{equation}
		\label{oneH}
		S_{D, L}^{-1}(s,T) = S_{D, R}^{-1}(s,T) = -4\mathcal{Q}_{c,s}^{-1}(T),
	\end{equation}
	and
	\begin{equation}
		\label{secondH}
		S_{\Delta D, L}^{-1}(s,T) = S_{\Delta D, R}^{-1}(s,T) = 16 \mathcal{Q}_{c,s}^{-2}(T).
	\end{equation}
	In the case $n = 5$, equation \eqref{FSconsequence} implies that for all $f \in \mathcal{SH}(U)$ we have
	$$
	D \Delta^{2}_{6} f(x) = 0.
	$$
	Hence, the functional calculus in \eqref{finsfun} is harmonic when $\beta = m = 1$, and biharmonic when $\beta = 1$, $m = 0$.
	
	While the operators in \eqref{oneH} and \eqref{secondH} differ from the resolvent operators related with the standard Cauchy kernels, we recover several properties of the resolvent equations within hyperholomorphic spectral theory. We use the resolvent equations obtained to derive counterparts of the Riesz projectors for the harmonic and biharmonic functional calculi. Moreover, we prove the product rules
	formulas for these functional calculi by combining the $S$-functional calculus with the functional calculus derived from \eqref{finsfun}.

	\section{Preliminaries}
	
	In this section, we recall some basic properties of slice hyperholomorphic functions, fine structures, and the $S$-functional calculus, which will be used
	in the paper.
	
	\begin{definition}
		Let $ U \subset \mathbb{R}^{n+1}$
		\begin{itemize}
			\item The set $U$ is called \emph{axially symmetric} if $[x]:= \{y \in \mathbb{R}^{n+1} \, : \, y= \hbox{Re}(x)+I|\underline{x}|, \quad I \in \mathbb{S}\} \subset U$ for any $x \in U$.
			\item The set $U$ is called \emph{slice domain} if $U \cap \mathbb{R} \neq \emptyset$ and if $U \cap \mathbb{C}_I$ is a domain in $ \mathbb{C}_I$ for any $I \in \mathbb{S}$.
		\end{itemize}
	\end{definition}

	\begin{definition}
		\label{axial}
		Let $U \subseteq \mathbb{R}^{n+1}$ be an axially symmetric domain. We define
		\begin{equation}
			\label{set}
			\mathcal{U} := \{ (u,v) \in \mathbb{R}^2 \; : \; u + Iv \in U \text{ for all } I \in \mathbb{S} \}.
		\end{equation}
		A function $f : U \to \mathbb{R}^{n+1}$ is called a left (resp.\ right) axial function (or slice function) if it admits the representation
		\begin{equation}
			f(x) = A(x_0, |\underline{x}|) + \underline{\omega}\, B(x_0, |\underline{x}|),
			\qquad
			\underline{\omega} = \frac{\underline{x}}{|\underline{x}|},
		\end{equation}
		where $A, B : \mathcal{U} \to \mathbb{R}_n$ satisfy the so called even-odd conditions
		\begin{equation}
			\label{eo}
			A(x_0, |\underline{x}|) = A(x_0, -|\underline{x}|),
			\qquad
			B(x_0, |\underline{x}|) = -B(x_0, -|\underline{x}|),
			\qquad
			\forall (x_0, |\underline{x}|) \in \mathcal{U}.
		\end{equation}
	\end{definition}

	\begin{definition}[Slice hyperholomorphic functions]
		\label{sh}
		Let $U \subseteq \mathbb{R}^{n+1}$ be an axially symmetric open set. A function $f: U \to \mathbb{R}^{n+1}$ is called left (resp. right) slice hyperholomorphic function if it is slice (see Definition~\ref{axial}) and can therefore be written as
		\begin{equation}
			\label{slice}
			f(x)=\alpha(u,v)+I \beta(u,v), \qquad \left(\, \hbox{resp.} \, \, f(x)= \alpha(u,v)+\beta(u,v)I\right), \quad \hbox{for} \quad x=u+Iv \in U
		\end{equation}
		where $ \alpha$, $\beta: \mathcal{U} \to \mathbb{R}_n$ satisfy the even-odd conditions, see \eqref{eo}, and the Cauchy-Riemann equations
		$$ \partial_u \alpha(u,v)=\partial_v \beta(u,v), \qquad \partial_v \alpha(u,v)=-\partial_u \beta(u,v).$$
		The set of left (resp. right) slice hyperholomorphic functions is denoted by $ \mathcal{SH}_L(U)$ (resp. $\mathcal{SH}_R(U)$). We say $f\in \mathcal{N}(U)$, called intrinsic slice hyperholomorphic, if the functions $\alpha$ and $\beta$ in \eqref{slice} are real-valued.
	\end{definition}
	
	\begin{definition}[Slice Cauchy domain]
		An axially symmetric open set $U \subset \mathbb{R}^{n+1}$ is said to be a slice Cauchy domain if, for every $I \in \mathbb{S}$, the set $U \cap \mathbb{C}_I$ forms a Cauchy domain in $\mathbb{C}_I$. In other words, $U$ is a slice Cauchy domain if, for each $I \in \mathbb{S}$, the set $\partial (U \cap \mathbb{C}_I)$ is a finite union of pairwise disjoint piecewise $C^1$ Jordan curves.
	\end{definition}
	\begin{definition}[Slice hyperholomorphic Cauchy kernels]
		\label{Ckernel}
		Let $x$, $s\in \mathbb{R}^{n+1}$ such that $x\not\in [s]$. We say that  $S_L^{-1}(s,x)$ (resp.  $S_R^{-1}(s,x)$) is written in the form I if
		$$
		S_L^{-1}(s,x):=-(x^2 -2x {\rm Re} (s)+|s|^2)^{-1}(x-\overline s), \quad \left(	\hbox{resp.} \, \,	S_R^{-1}(s,x):=-(x-\bar s)(x^2-2{\rm Re}(s)x+|s|^2)^{-1} \right).
		$$
		We say that $S_L^{-1}(s,x)$ (resp.  $S_R^{-1}(s,x)$) is written in the form II if
		$$
		S_L^{-1}(s,x):=(s-\bar x)(s^2-2{\rm Re}(x) s+|x|^2)^{-1}, \quad \left(\hbox{resp.} \, \,S_R^{-1}(s,x):=(s^2-2{\rm Re}(x)s+|x|^2)^{-1}(s-\bar x)\right).
		$$
	\end{definition}

	\begin{theorem}[Cauchy-formulas for slice hyperholomorphic functions]
		Let $U \subset \mathbb{R}^{n+1}$ be a bounded slice Cauchy domain, let $I \in \mathbb{S}$, and set $ds_I = ds(-I)$. If $f$ is a left (resp.\ right) slice hyperholomorphic function on a domain containing $\overline{U}$ and $x \notin [s]$, then
		\begin{equation}
			\label{integrals}
			f(x)= \frac{1}{2 \pi} \int_{\partial(U \cap \mathbb{C}_I)} S^{-1}_L(s,x)ds_I f(s), \qquad \bigg( \hbox{resp.} \, \, f(x)= \frac{1}{2 \pi} \int_{\partial(U \cap \mathbb{C}_I)}  f(s) ds_IS^{-1}_R(s,x)\bigg), \quad \forall x \in U,
		\end{equation}
		where $S^{-1}_L(s,x)$ (resp.\ $S^{-1}_R(s,x)$) denotes the left (resp.\ right) slice hyperholomorphic Cauchy kernel.
		Moreover, the integrals in \eqref{integrals} are independent of both the domain $U$ and the choice of the imaginary unit $I \in \mathbb{S}$.
	\end{theorem}
	
	For further details on the theory of slice hyperholomorphic functions, we refer the reader to \cite{CGK, ColomboSabadiniStruppa2011, CSS2}. In hypercomplex analysis, another prominent class of functions is the following.
	
	\begin{definition}
		\label{am}
		Let $U \subseteq \mathbb{R}^{n+1}$ be an axially symmetric domain.
		A function $f : U \to \mathbb{R}^{n+1}$ of class $\mathcal{C}^1$ is called
		left (resp.\ right) axially monogenic if it belongs to the kernel of the Dirac operator, that is,
		$$
		Df(x)=\bigg(\frac{\partial}{\partial x_0}
		+ \sum_{i=1}^{n} e_i \frac{\partial}{\partial x_i}\bigg) f(x)=0, \quad
		\bigg( \hbox{resp.} \, \,
		f(x)D=\frac{\partial}{\partial x_0} f(x)
		+ \sum_{i=1}^{n} \frac{\partial}{\partial x_i} f(x)\, e_i=0 \bigg),
		$$
		and if it is of axial form; see Definition~\ref{axial}. We denote the class of left (resp right) axially monogenic functions by $ \mathcal{AM}_L(U)$ (resp. $ \mathcal{AM}_R(U)$).
	\end{definition}
	
	\begin{remark}
		In the rest of the paper, whenever it is not necessary to distinguish between the two cases, we will restrict ourselves to the left case.
	\end{remark}
	
	For further information on this class of functions, we refer the reader to \cite{red, green, GHS}. The following  fundamental result in hypercomplex analysis establishes a connection between these function spaces  $\mathcal{SH}_L(U)$  and $\mathcal{AM}_L(U)$ (cf. \cite{ColSabStrupSce, Sce}).
	
	\begin{theorem}[Fueter-Sce mapping theorem]\label{FSTheo}
		Let $n\in\mathbb{N}$ be odd and set $h_n:= \frac{n-1}{2}$. Suppose that $f_0(z)=\alpha(u,v)+i \beta(u,v)$ is a holomorphic function defined in a domain $\Pi$ in the upper half complex plane. Let
		$ U_{\Pi}:= \{x=x_0+\underline{x} \, : \, (x_0, |\underline{x}|) \in \Pi\},$
		be the open set induced by $\Pi$ in $\mathbb{R}^{n+1}$. The slice operator, denoted by $T_{FS1}$, and defined by
		$ f(x)=T_{FS1}(f_0)=\alpha(x_0, |\underline{x}|)+ \frac{\underline{x}}{|\underline{x}|} \beta(x_0, |\underline{x}|), \quad x \in U_\Pi$
		maps the set of holomorphic functions in the set of left slice hyperholomorphic functions. Moreover, the operator
		$T_{FS2}:= \Delta_{n+1}^{h_n}$
		maps $f(x)=T_{FS1}(f_0)$ into the set of left axially monogenic functions.
	\end{theorem}

	The Dirac fine structures on the $S$-spectrum  are based on the factorization of the operator $T_{FS2}:=\Delta_{n+1}^{h_n}$:
	\begin{equation}
		\label{sFS}
		\begin{tikzcd}[column sep=3em, ampersand replacement=\&]
			\mathcal{SH}_L(U)  \arrow[r, "\Delta_{n+1}^{h_n}"] \& \mathcal{AM}_L(U).
		\end{tikzcd}
	\end{equation}

	\begin{definition}[Dirac fine structures on the $S$-spectrum]
		\label{fine}
		Let $n$ be an odd integer and denote $h_n := \frac{n-1}{2}$.
		The Dirac fine structures on the $S$-spectrum consist of:
		\begin{itemize}
			\item All axial functions belonging to the kernels of the operators
			$$
			D^\beta \Delta_{n+1}^m, \quad \hbox{and} \quad  \overline{D}^\beta \Delta_{n+1}^m \qquad \text{with }\beta + m \leq h_n.
			$$
			
			\item Their associated functional calculi, defined through integral representations based on the Cauchy formula for slice hyperholomorphic functions.
		\end{itemize}
		In particular, the fine structures arising from the operators $D^\beta \Delta_{n+1}^m$  are called $D$-fine structures. Similarly, the factorizations obtained through the operators $\overline{D}^\beta \Delta_{n+1}^m$ are referred to as $\overline{D}$-fine structures.
	\end{definition}
	
	The investigation of the Dirac fine structures was initiated in \cite{CDPS} and was subsequently continued in \cite{CDP25, CDP2026, Fivedim, CDS} for bounded operators, and in \cite{CDP232, CPS1, DPS} for unbounded operators, see \cite{AD} for a survey.
	We now focus on the case $n=5$ (i.e., $h_n=2$). Specifically, we consider two particular $D$-fine structures.  The first possibility is to set $m=1$ and $\beta=1$. In this case, the scheme \eqref{sFS} can be factorized as
	
	\begin{equation}
		\label{12}
		\begin{tikzcd}[column sep=3em, ampersand replacement=\&]
			\mathcal{SH}_L(U) \arrow[r, "D\Delta_6"] \& \mathcal{AH}^L(U) \arrow[r, "\Delta_6"] \& \mathcal{AM}_L(U)
		\end{tikzcd}
	\end{equation}
	
	where $\mathcal{AH}^L(U)$ denotes the set of axially harmonic functions, defined by
	
	\begin{equation}
		\label{A1}
		\mathcal{AH}^L(U) = D \Delta_6 \big( \mathcal{SH}_L(U) \big)
		= \big\{ h \in C^\infty(U) \; | \; h = D \Delta_6 f, \ f \in \mathcal{SH}_L(U) \big\}.
	\end{equation}
	
	A function $h \in \mathcal{AH}^L(U)$ is harmonic by virtue of the Fueter–Sce mapping theorem, since $\Delta_6 h = \Delta_6 (D \Delta_6 f) = D \Delta_6^2 f = 0.$
	\\The second possible $D$-fine structure is obtained by taking $m=0$ and $\beta=1$. In this case, the scheme \eqref{sFS} factorizes as
	\begin{equation}
		\label{13}
		\begin{tikzcd}[column sep=3em, ampersand replacement=\&]
			\mathcal{SH}_L(U) \arrow[r, "D"] \& \mathcal{APH}_{2}^L(U) \arrow[r, "\Delta_6"] \& \mathcal{AM}_L(U)
		\end{tikzcd}
	\end{equation}
	where $\mathcal{APH}_{2}^L(U)$ denotes the set of axially polyharmonic functions of order 2, defined by
	
	\begin{equation}
		\label{A2}
		\mathcal{APH}_{2}^L(U) = D \big( \mathcal{SH}_L(U) \big)
		= \big\{ g \in C^\infty(U) \; | \; g = D f, \ f \in \mathcal{SH}_L(U) \big\}.
	\end{equation}
	
	A function $g \in \mathcal{APH}_{2}(U)$ is polyharmonic of order 2; indeed, by the Fueter mapping theorem, we have  $\Delta_6^2 g = \Delta_6^2 (D f) = D \Delta_6^2 f = 0.$
	\newline
	\newline
	All functional calculi associated with the fine structures are based on the notion of the $S$-spectrum. We first introduce some notations.
	Let $\mathcal{B}(V_n)$ be the Banach space of all bounded right linear Clifford operators
	$T=\sum_A e_A T_A$ where $A = \{\ell_1 \dots \ell_r\} \in \mathcal{P}\{1, 2, \dots, n\}$, $\ell_1 < \dots < \ell_r$ is a multi-index, $T_\emptyset = T_0$, acting on a two sided Clifford Banach module $V_n = V \otimes \mathbb{R}_n$, where $V$ is a real Banach space.
	By $\mathcal{BC}(V_n)$ we denotes the subset of $\mathcal{B}(V_n)$ consisting of Clifford operators with commuting components, i.e., operators of the type $\sum_A e_A T_A$ where $T_A$ commute pairwise and subclass $\mathcal{BC}^{0,1}(V_n)$ that consists of paravector operators with commuting components, i.e., $T = T_0 + e_1 T_1 + \dots + e_n T_n$.

	\medskip
	To introduce the $S$-functional calculus, we first recall the concept of the spectrum for Clifford operators. If $T \in \mathcal{BC}^{0,1}(V_n)$, the F-spectrum of $T$ is defined as
	\[
	\sigma_F(T) = \{s \in \mathbb{R}^{n+1} : s^2 \mathcal{I} - s(T + \overline{T}) + T\overline{T} \text{ is not invertible in } \mathcal{B}(V_n)\},
	\]
	where we have set $\overline{T} := T_0 - e_1 T_1 - \dots - e_n T_n$, and the $F$-resolvent set $\rho_F(T) := \mathbb{R}^{n+1} \setminus \sigma_F(T).$
	\begin{remark}
		The F-spectrum is the commutative version of the $S$-spectrum (cf. \cite{ColSab}), i.e., we have
		\[
		\sigma_F(T) = \sigma_S(T), \ \ \ \text{ for }\ \ \ T \in \mathcal{BC}^{0,1}(V_n),
		\]
		where
		the S-spectrum of the operator $T$, which is suitable for $T\in\mathcal{B}(V_n)$ is defined as
		\[
		\sigma_S(T) = \{s \in \mathbb{R}^{n+1} : T^2 - 2\text{Re}(s)T + |s|^2 I \text{ is not invertible in } \mathcal{B}(V_n)\}.
		\]
		Analogously $\rho_S(T) = \mathbb{R}^{n+1} \setminus \sigma_S(T)$ is the $S$-resolvent set.
	\end{remark}
	For $T \in \mathcal{BC}^{0,1}(V_n)$, the commutative version of left (resp, right) S-resolvent operators are defined as
	\begin{equation} \label{eq:1.7}
		S_L^{-1}(s,T) := (s\mathcal{I} - \overline{T})	Q_{c,s}^{-1}(T),  \qquad \left(\hbox{resp.} \, \, S_R^{-1}(s,T) := 	Q_{c,s}^{-1}(T)(s \mathcal{I} - \overline{T}) \,\right), \quad s \in \rho_S(T).
	\end{equation}
	where
	\begin{equation} \label{eq:1.9}
		Q_{c,s}^{-1}(T) := (s^2 \mathcal{I} - s(T + \overline{T}) + T\overline{T})^{-1}, \quad s \in \rho_S(T),
	\end{equation}
	is the commutative pseudo S-resolvent operator. The $S$-resolvent operators satisfied the following properties (cf. \cite{ColomboSabadiniStruppa2011}):
	\begin{equation}
		\label{sl}
		S^{-1}_L(s,T)s - T S_{L}^{-1}(s,T) = \mathcal{I},
	\end{equation}
	\begin{equation}
		\label{sr}
		s S^{-1}_R(s,T) - S_{R}^{-1}(s,T) T = \mathcal{I},
	\end{equation}
	for $T \in \mathcal{BC}^{0,1}(V_n)$ and $s \in \rho_S(T)$.
	\begin{remark}
		In the literature, the notation for the commutative version of the left $S$-resolvent is usually $S_{c,L}^{-1}(s,T)$ (resp. $S_{c,R}^{-1}(s,T)$). However, for the sake of simplicity, and since no confusion arises, in this paper we use the notation $S_L^{-1}(s,T)$ (resp. $S_R^{-1}(s,T)$).
	\end{remark}
	\begin{definition}
		Let $T \in \mathcal{BC}^{0,1}(V_n)$, and let $U \subset \mathbb{R}^{n+1}$ be a bounded slice Cauchy domain. We denote by $\mathcal{SH}^L_{\sigma_S(T)}(U)$, $\mathcal{SH}^R_{\sigma_S(T)}(U)$, and $\mathcal{N}_{\sigma_S(T)}(U)$ the sets of all left, right, and intrinsic slice hyperholomorphic functions $f$ such that $\sigma_S(T) \subset U \subset \hbox{dom}(f)$, where $\hbox{dom}(f)$ is the domain of the function $f$.
	\end{definition}

	With these tools, we are now ready to define the $S$-functional calculus.

	\begin{definition}[Commutative $S$-functional calculus]
		\label{Sfun}
		Let $T \in T \in \mathcal{BC}^{0,1}(V_n)$, and let $U \subset \mathbb{R}^{n+1}$ be a bounded slice Cauchy domain. Let $I \in \mathbb{S}$ and define $ds_I = ds(-I)$. For any $f \in \mathcal{SH}^L_{\sigma_S(T)}(U)$ (or $f \in \mathcal{SH}^R_{\sigma_S(T)}(U)$, respectively), we set
		\begin{equation}
			f(T)=\frac{1}{2 \pi} \int_{\partial(U \cap \mathbb{C}_I)} S^{-1}_{L}(s,T)ds_If(s), \quad \left( \hbox{resp.} \, \, f(T)=\frac{1}{2 \pi} \int_{\partial(U \cap \mathbb{C}_I)}f(s) ds_I S^{-1}_{R}(s,T)\right).
		\end{equation}
	\end{definition}
	The $S$-functional calculus, also in its noncommutative formulation,
	satisfies a product rule, which can be proved using the resolvent equations for the $S$-resolvent operators, see \cite{ACDS, ACGS}.

	\begin{theorem}[Generalized $S$-resolvent equation]
		\label{Bres}
		Let $T \in \mathcal{BC}^{0,1}(V_n)$  and $B \in \mathcal{B}(V_n)$ be such that $B$ commutes with $T$, then, For $p$, $s\in\rho_S(T)$ with $s\not\in[p]$ , we have
		\begin{equation}
			\label{Breseq}
			S^{-1}_R(s,T)BS_{L}^{-1}(p,T)=[(S^{-1}_R(s,T)B-BS_L^{-1}(p,T))p] -\bar{s}(S^{-1}_R(s,T)B-BS_L^{-1}(p,T))] Q_s^{-1}(p),
		\end{equation}
		where $Q_s(p):=p^2-2s_0p+|s|^2$.
	\end{theorem}
	
	\begin{remark}
		\label{listS}
		We summarize below the principal properties of the generalized $S$-resolvent equation.
		\begin{itemize}
			\item It preserves right slice hyperholomorphicity in the variable $s$ and left slice hyperholomorphicity in the variable $p$.
			\item The product $S^{-1}_R(s,T) B S^{-1}_L(p,T)$ is transformed into the difference $S^{-1}_R(s,T) B - B S^{-1}_L(p,T)$.
			\item The difference $S^{-1}_R(s,T) B - B S^{-1}_L(p,T)$ is entangled with the slice hyperholomorphic Cauchy kernel
			written in the non-commutative version, i.e., in Form I, see Definition \ref{Ckernel}.
		\end{itemize}
	\end{remark}

	\section{The resolvent equation and the product rule for the bi-harmonic functional calculus}
	
	In this section, we establish a resolvent equation, a product rule and a counter part of the Riesz projectors for the biharmonic functional calculus arising from the factorization in \eqref{13}. This functional calculus was established in \cite{Fivedim}, where the Dirac operator $D$ was applied to the left and right slice hyperholomorphic Cauchy kernels (see \eqref{Ckernel}), obtaining
	\begin{equation}\label{eq}
		D S_L^{-1}(s,x) = D S_R^{-1}(s,x) = -4 \mathcal{Q}_{c,s}^{-1}(x).
	\end{equation}
	By combining \eqref{eq} with the slice hyperholomorphic Cauchy formulas, one obtains integral representation formulas for functions in $\mathcal{APH}_2^L(U)$ (see \eqref{A2}), which lead to the notion of biharmonic functional calculi.
	
	\begin{definition}[Bi-Harmonic functional calculus based on the $S$-spectrum]
		\label{bih1}
		Let $T \in \mathcal{BC}^{0,1}(V_5)$ be such that its components $T_i$, $i=1,2,3,4,5$, have real spectra and $T_0=0$. Let $U$ be a bounded slice Cauchy domain, and set $ds_I = ds(-I)$ for $I \in \mathbb{S}$.
		For $f \in \mathcal{SH}_{\sigma_S(T)}^L(U)$ (resp.\ $f \in \mathcal{SH}^R_{\sigma_S(T)}(U)$), let $f_D(x) = Df(x)$ (resp.\ $f_D(x) = f(x)D$). We define the biharmonic functional calculus based on the $S$-spectrum for the operator $T$ as follows:
		\begin{equation}
			\label{bih}
			f_D(T)= \frac{1}{2 \pi} \int_{\partial(U \cap \mathbb{C}_I)} S_{D}^{-1}(s,T)ds_I f(s), \qquad \left(\hbox{resp.} \, \, f_D(T)= \frac{1}{2 \pi} \int_{\partial(U \cap \mathbb{C}_I)} f(s)ds_IS_{D}^{-1}(s,T) \right)
		\end{equation}	
		where the resolvent operator $S_{D}^{-1}(s,T)$ is defined as
		\begin{equation}
			\label{H1}
			S_{D,L}^{-1}(s,T)=S_{D,R}^{-1}(s,T) = -4\mathcal{Q}_{c,s}^{-1}(T).
		\end{equation}
	\end{definition}
	\begin{remark}\label{FCOMPON}
		We observe that in Definition \ref{bih1} it is possible to assume $T_0 \neq 0$ when one of the components $T_i$, for $i=1, \dots, n$, is the zero operator. The requirement that the components $T_i$ (for $i=0,1, \dots, 5$) have real spectra is necessary to ensure the compatibility of the $F$-functional calculus (which is part of the fine structure) with the monogenic functional calculus developed by McIntosh and collaborators, see Remark 3.31 in \cite{CDP25}.
	\end{remark}
	Another functional calculus arising from the fine structures (see Definition \ref{fine}), corresponding to $\beta = 0$ and $m = 1$, which we will need in the sequel, is the following; see \cite{CDP25, Fivedim}.
	\begin{definition}[$\Delta$-functional calculus based on the $S$-spectrum]
		\label{delta}
		Under the same assumptions on the operator $T$ and on the set $U$ as in Definition~\ref{bih1}, let $f \in \mathcal{SH}_{\sigma_S(T)}^L(U)$ (resp.\ $f \in \mathcal{SH}^R_{\sigma_S(T)}(U)$), and define
		$
		f_{\Delta_6}(x) := \Delta_6 f(x).
		$
		The $\Delta$-functional calculus based on the $S$-spectrum of the operator $T$ is defined by:
		\begin{equation}
			f_{\Delta}(T)= \frac{1}{2 \pi} \int_{\partial(U \cap \mathbb{C}_I)} S_{\Delta,L}^{-1}(s,T)ds_I f(s), \qquad \left(\hbox{resp.} \, \, f_{\Delta}(T)= \frac{1}{2 \pi} \int_{\partial(U \cap \mathbb{C}_I)} f(s)ds_IS_{\Delta, R}^{-1}(s,T) \right)
		\end{equation}	
		where the resolvent operators are given by
		\begin{align}\label{K-resol.op}
			S^{-1}_{\Delta,L}(s,T)=-8 S_L^{-1}(s,T)\mathcal{Q}^{-1}_{c,s}(T),\qquad \left(\hbox{resp.} \, \, S^{-1}_{\Delta, R}(s,T)=-8 \mathcal{Q}^{-1}_{c,s}(T)S_R^{-1}(s,T) \right).
		\end{align}
	\end{definition}
	
	The above resolvent operators satisfy the following property

	\begin{proposition}\label{Delta equation}
		Let $T \in \mathcal{BC}^{0,1}(V_5)$ and let $s \in \rho_S(T)$. Then we have
		\begin{equation}
			\label{rK1}
			S^{-1}_{\Delta,L}(s,T)s-TS^{-1}_{\Delta,L}(s,T)=-8 \mathcal{Q}_{c,s}^{-1}(T)
		\end{equation}
		and
		\begin{equation}
			\label{rK2}
			sS^{-1}_{\Delta,R}(s,T)-S^{-1}_{\Delta,R}(s,T)T=-8 \mathcal{Q}_{c,s}^{-1}(T).
		\end{equation}
	\end{proposition}
	\begin{proof}
		We prove only \eqref{rK1}, since \eqref{rK2} can be established by analogous arguments. By the definition of $S^{-1}_{\Delta,L}(s,T)$ and \eqref{sl}, we obtain
		$$ S^{-1}_{\Delta,L}(s,T)s-TS^{-1}_{\Delta,L}(s,T)=-8[S^{-1}_L(s,T)s-T S^{-1}_L(s,T)] \mathcal{Q}_{c,s}^{-1}(T)=-8 \mathcal{Q}_{c,s}^{-1}(T).\hfill\qedhere$$
	\end{proof}

	To derive the resolvent equation for the biharmonic functional calculus, which satisfies properties analogous to those of the generalized $S$-resolvent equation (see Remark~\ref{listS}), we require some technical lemmas.

	\begin{lemma}\label{Lemmastep0}
		Let $T \in \mathcal{BC}^{0,1}(V_5)$. For $p$, $s\in\rho_S(T)$ with $s\not\in[p]$ we have
		$$ Q_s^{-1}(p)\mathcal{Q}_{c,p}^{-m}(T)=\mathcal{Q}_{c,p}^{-m}(T)Q_s^{-1}(p), \qquad m \in \mathbb{N},$$
		where $Q_s(p)=p^2-2s_0p+|s|^2$.
	\end{lemma}
	\begin{proof}
		The claim follows from the fact that $\mathcal{Q}_{c,s}^{-m}(T)$ is a $\mathcal{B}(V_5)$-valued intrinsic slice hyperholomorphic function in $p$ (see \cite[Prop.~6.15]{CDP2026}), and that the polynomial $Q_s^{-1}(p)$ is intrinsic slice hyperholomorphic in the same variable.
	\end{proof}
	
	\begin{lemma}\label{Lemmastep1}
		Let $T\in \mathcal{BC}^{0,1}(V_5)$. For $p$, $s\in\rho_S(T)$ with $s\not\in[p]$ , the following equation holds
		\begin{align}\label{D-step 3}
			\begin{split}
				&S^{-1}_R(s,T)\mathcal{Q}_{c,p}^{-1}(T)=[(S^{-1}_R(s,T)\mathcal{Q}_{c,p}^{-1}(T)p-S^{-1}_R(s,T)T\mathcal{Q}_{c,p}^{-1}(T)-\mathcal{Q}_{c,p}^{-1}(T))p\\
				& -\bar{s}(S^{-1}_R(s,T)\mathcal{Q}_{c,p}^{-1}(T)p-S^{-1}_R(s,T)T\mathcal{Q}_{c,p}^{-1}(T)-\mathcal{Q}_{c,p}^{-1}(T))] Q_s^{-1}(p).
			\end{split}
		\end{align}
		\begin{align}\label{D-step 6}
			\begin{split}
				&\mathcal{Q}_{c,s}^{-1}(T)S_{L}^{-1}(p,T)=[(\mathcal{Q}_{c,s}^{-1}(T)-s\mathcal{Q}_{c,s}^{-1}(T)S_L^{-1}(p,T)+\mathcal{Q}_{c,s}^{-1}(T)TS_{L}^{-1}(p,T))p\\
				& -\bar{s}(\mathcal{Q}_{c,s}^{-1}(T)-s\mathcal{Q}_{c,s}^{-1}(T)S_L^{-1}(p,T)+\mathcal{Q}_{c,s}^{-1}(T)TS_{L}^{-1}(p,T))] Q_s^{-1}(p).
			\end{split}
		\end{align}
	\end{lemma}
	\begin{proof}
		We begin by proving equation \eqref{D-step 3}. Let $B = T$ in the generalized $S$-resolvent equation (cf. Theorem \ref{Bres}). Multiplying it on the right by $8\mathcal{Q}_{c,p}^{-1}(T)$ and using \eqref{K-resol.op} together with Lemma \ref{Lemmastep0}, we obtain
		\begin{align}\label{D-step 1}
			\begin{split}
				-S^{-1}_R(s,T)TS^{-1}_{\Delta,L}(p,T)&=[(8S^{-1}_R(s,T)T\mathcal{Q}_{c,p}^{-1}(T)+TS^{-1}_{\Delta,L}(p,T))p\\
				&-\bar{s}(8S^{-1}_R(s,T)T\mathcal{Q}_{c,p}^{-1}(T)+TS^{-1}_{\Delta,L}(p,T))] Q^{-1}_s(p).
			\end{split}
		\end{align}
		Next, we consider the case $B = I$ in Theorem \ref{Bres}. Multiplying the resulting equation on the right by $-8\mathcal{Q}_{c,p}^{-1}(T)p$ and using \eqref{K-resol.op} together with Lemma \ref{Lemmastep0}, we get
		\begin{align}\label{D-step 2}
			\begin{split}
				S^{-1}_R(s,T)S^{-1}_{\Delta,L}(p,T)p=&[(-8S^{-1}_R(s,T)\mathcal{Q}_{c,p}^{-1}(T)p-S^{-1}_{\Delta,L}(p,T)p)p\\
				&-\bar{s}(-8S^{-1}_R(s,T)\mathcal{Q}_{c,p}^{-1}(T)p-S^{-1}_{\Delta,L}(p,T)p)] Q_s^{-1}(p).
			\end{split}
		\end{align}
		By adding \eqref{D-step 1} and \eqref{D-step 2} and using \eqref{rK1} on both the right- and left-hand sides, we obtain
		\begin{align*}
			&S^{-1}_R(s,T)\mathcal{Q}_{c,p}^{-1}(T)=[(S^{-1}_R(s,T)\mathcal{Q}_{c,p}^{-1}(T)p-S^{-1}_R(s,T)T\mathcal{Q}_{c,p}^{-1}(T)-\mathcal{Q}_{c,p}^{-1}(T))p\\
			& -\bar{s}(S^{-1}_R(s,T)\mathcal{Q}_{c,p}^{-1}(T)p-S^{-1}_R(s,T)T\mathcal{Q}_{c,p}^{-1}(T)-\mathcal{Q}_{c,p}^{-1}(T))] Q_s^{-1}(p).
		\end{align*}	
		Regarding equation \eqref{D-step 6}, one uses the same techniques; the only difference being that, instead of right multiplication, one uses left multiplication.
	\end{proof}
	\begin{lemma}\label{D-midstep}
		Let $T\in \mathcal{BC}^{0,1}(V_5)$. For $p$, $s\in\rho_S(T)$ with $s\not\in[p]$, the following relations hold:
		\begin{align}
			\label{D-step 8}
			\mathcal{Q}_{c,s}^{-1}(T)T\mathcal{Q}_{c,p}^{-1}(T)&=[(\mathcal{Q}_{c,s}^{-1}(T)TS_{L}^{-1}(p,T)-S^{-1}_R(s,T)T\mathcal{Q}_{c,p}^{-1}(T))p\\
			\nonumber
			&-\bar{s}(\mathcal{Q}_{c,s}^{-1}(T)TS_{L}^{-1}(p,T)-S^{-1}_R(s,T)T\mathcal{Q}_{c,p}^{-1}(T))]Q_s^{-1}(p).
		\end{align}
		\begin{align}
			\label{D-step 81}
			\mathcal{Q}_{c,s}^{-1}(T)\overline{T}\mathcal{Q}_{c,p}^{-1}(T)&=-[(S^{-1}_R(s,T)\mathcal{Q}_{c,p}^{-1}(T)p-s\mathcal{Q}_{c,s}^{-1}(T)S_L^{-1}(p,T))p\\
			\nonumber
			&-\bar{s}(S^{-1}_R(s,T)\mathcal{Q}_{c,p}^{-1}(T)p-s\mathcal{Q}_{c,s}^{-1}(T)S_L^{-1}(p,T))]Q_s^{-1}(p).
		\end{align}
	\end{lemma}
	\begin{proof}
		We first observe that, by the definition of the left and right $S$-resolvent operators, we have
		\begin{align*}
			\mathcal{Q}_{c,s}^{-1}(T)T S^{-1}_L(p,T)-S^{-1}_R(s,T)T \mathcal{Q}_{c,p}^{-1}(T)&= \mathcal{Q}_{c,s}^{-1}(T)T(p \mathcal{I}-\overline{T}) \mathcal{Q}_{c,p}^{-1}(T)- \mathcal{Q}_{c,s}^{-1}(T)(s \mathcal{I}-\overline{T})T \mathcal{Q}_{c,p}^{-1}(T)\\
			&= \mathcal{Q}_{c,s}^{-1}(T)(Tp -sT) \mathcal{Q}_{c,p}^{-1}(T).
		\end{align*}
		This implies
		\begin{align*}
			\begin{split}
				&(\mathcal{Q}_{c,s}^{-1}(T)TS_{L}^{-1}(p,T)-S^{-1}_R(s,T)T\mathcal{Q}_{c,p}^{-1}(T))p-\bar{s}(\mathcal{Q}_{c,s}^{-1}(T)TS_{L}^{-1}(p,T)-S^{-1}_R(s,T)T\mathcal{Q}_{c,p}^{-1}(T))\\
				&=\mathcal{Q}_{c,s}^{-1}(T)( (Tp-sT)p-\bar{s}(Tp -sT))\mathcal{Q}_{c,p}^{-1}(T)\\
				&= \mathcal{Q}_{c,s}^{-1}(T)(Tp^2-(s+\bar{s})Tp +T|s|^2)\mathcal{Q}_{c,p}^{-1}(T)\\
				&=\mathcal{Q}_{c,s}^{-1}(T)T\mathcal{Q}_{c,p}^{-1}(T)Q_s(p).
			\end{split}
		\end{align*}
		This proves \eqref{D-step 8}. To show \eqref{D-step 81}, we observe
		\begin{align*}
			S^{-1}_R(s,T)\mathcal{Q}_{c,p}^{-1}(T)p-s\mathcal{Q}_{c,s}^{-1}(T)S_L^{-1}(p,T)&= \mathcal{Q}_{c,s}^{-1}(T)(s\mathcal{I}-\overline{T}) \mathcal{Q}_{c,p}^{-1}(T)p-s \mathcal{Q}_{c,s}^{-1}(T)(p \mathcal{I}-\overline{T}) \mathcal{Q}_{c,p}^{-1}(T)\\
			&= \mathcal{Q}_{c,s}^{-1}(T)(s \overline{T}-\overline{T}p)\mathcal{Q}_{c,p}^{-1}(T).
		\end{align*}
		This implies that
		\begin{align*}
			\begin{split}
				&(S^{-1}_R(s,T)\mathcal{Q}_{c,p}^{-1}(T)p-s\mathcal{Q}_{c,s}^{-1}(T)S_L^{-1}(p,T))p-\bar{s}(S^{-1}_R(s,T)\mathcal{Q}_{c,p}^{-1}(T)p-s\mathcal{Q}_{c,s}^{-1}(T)S_L^{-1}(p,T))\\
				&= \mathcal{Q}_{c,s}^{-1}(T)((s\overline{T}-\overline{T}p)p-\bar{s}(s\overline{T}-\overline{T}p))\mathcal{Q}_{c,p}^{-1}(T)\\
				&= -\mathcal{Q}_{c,s}^{-1}(T)\overline{T}\mathcal{Q}_{c,p}^{-1}(T)Q_s(p).\hfill\qedhere
			\end{split}
		\end{align*}
	\end{proof}
	\begin{theorem}\label{D-resolvent eq.}
		Let $T\in \mathcal{BC}^{0,1}(V_5)$. For $p$, $s\in\rho_S(T)$ with $s\not\in[p]$, we have the following equation:
		\begin{align}
			\nonumber
			S^{-1}_R(s,T)\mathcal{Q}_{c,p}^{-1}(T)&+\mathcal{Q}_{c,s}^{-1}(T)S_{L}^{-1}(p,T)-  2\mathcal{Q}_{c,s}^{-1}(T)\underline{T}\mathcal{Q}_{c,p}^{-1}(T)
			\\
			\label{preres}
			&=[(\mathcal{Q}_{c,s}^{-1}(T)-\mathcal{Q}_{c,p}^{-1}(T))p -\bar{s}(\mathcal{Q}_{c,s}^{-1}(T)-\mathcal{Q}_{c,p}^{-1}(T))] Q^{-1}_s(p),
		\end{align}
		where $\underline{T}= T_1 e_1+T_2 e_2+T_3 e_3+T_4 e_4+T_5 e_5$.
	\end{theorem}
	\begin{proof} By making the sum of \eqref{D-step 3} and \eqref{D-step 6} we obtain
		\begin{align*}
			\begin{split}
				&S^{-1}_R(s,T)\mathcal{Q}_{c,p}^{-1}(T)+\mathcal{Q}_{c,s}^{-1}(T)S_{L}^{-1}(p,T)\\
				&=[(\mathcal{Q}_{c,s}^{-1}(T)-\mathcal{Q}_{c,p}^{-1}(T))p -\bar{s}(\mathcal{Q}_{c,s}^{-1}(T)-\mathcal{Q}_{c,p}^{-1}(T))] Q_s^{-1}(p)\\
				&+[(\mathcal{Q}_{c,s}^{-1}(T)TS_{L}^{-1}(p,T)-S^{-1}_R(s,T)T\mathcal{Q}_{c,p}^{-1}(T))p-\bar{s}(\mathcal{Q}_{c,s}^{-1}(T)TS_{L}^{-1}(p,T)-S^{-1}_R(s,T)T\mathcal{Q}_{c,p}^{-1}(T))]Q_s^{-1}(p)\\
				&+[(S^{-1}_R(s,T)\mathcal{Q}_{c,p}^{-1}(T)p
				-s\mathcal{Q}_{c,s}^{-1}(T)S_L^{-1}(p,T))p-\bar{s}(S^{-1}_R(s,T)\mathcal{Q}_{c,p}^{-1}(T)p
				-s\mathcal{Q}_{c,s}^{-1}(T)S_L^{-1}(p,T))]Q_s^{-1}(p).
			\end{split}
		\end{align*}
		The result follows by using \eqref{D-step 8} and \eqref{D-step 81} and the fact that $T-\overline{T}=2 \underline{T}$.
	\end{proof}
	\begin{theorem}[Resolvent equation for the bi-harmonic functional calculus based on the $S$-spectrum]\label{corol D-resolent eq.}  Let $T\in \mathcal{BC}^{0,1}(V_5)$. For $p$, $s\in\rho_S(T)$ with $s\not\in[p]$, we have the following equation:
		\begin{align}\label{corol D-resolent}
			\begin{split}
				&S_{D,R}^{-1} (s,T) S_{L}^{-1} (p,T)+S_{R}^{-1} (s,\overline{T}) S_{D,L}^{-1} (p,T)\\
				&=[(S^{-1}_{D,L}(s,T)-(S^{-1}_{D,L}(p,T))p -\bar{s}((S^{-1}_{D,L}(s,T)-(S^{-1}_{D,L}(p,T))] Q_s^{-1}(p).
			\end{split}
		\end{align}
	\end{theorem}
	\begin{proof}
		The result follows from Theorem \ref{D-resolvent eq.} since
		\begin{align*}
			S^{-1}_R(s,T)\mathcal{Q}_{c,p}^{-1}(T)-  2\mathcal{Q}_{c,s}^{-1}(T)\underline{T}\mathcal{Q}_{c,p}^{-1}(T)
			&= \mathcal{Q}_{c,s}^{-1}(T)(s-\overline{T} -2\underline{T})\mathcal{Q}_{c,p}^{-1}(T)\\
			&=\mathcal{Q}_{c,s}^{-1}(T)(s-T)\mathcal{Q}_{c,p}^{-1}(T)\\
			&= S^{-1}_R(s,\overline{T})\mathcal{Q}_{c,p}^{-1}(T),
		\end{align*}
		and by multiplying both sides of the equation \eqref{preres} by $-4$ we get the equation \eqref{corol D-resolent}.
	\end{proof}

	\begin{theorem}[Product rule for the bi-harmonic functional calculus]\label{D prod. rule} Let $T\in \mathcal{BC}^{0,1}(V_5)$ be such that its components $T_i$, $i=1,2,3,4,5$ have real spectra and $T_0=0$. If $f\in\mathcal{N}_{\sigma_S(T)}(U)$ and $g\in\mathcal{SH}^L_{\sigma_S(T)}(U)$, then we have
		\begin{equation}
			(fg)_D(T)=f_D(T)g(T)+f(\overline{T})g_D(T).
		\end{equation}
	\end{theorem}
	\begin{proof}
		
		Let $G_1$ and $G_2$ be two bounded slice Cauchy domains containing the $S$-spectrum of $T$, $\overline{G}_1\subset G_2$ and $\overline{G}_2\subset \text{dom}(f)\cap \text{dom}(g)$. Choose $p\in \partial(G_1\cap\mathbb{C}_I)$ and $s\in \partial(G_2\cap\mathbb{C}_I)$. For every $I\in\mathbb{S}$, it follows from the definition of the bi-harmonic functional calculus, see  \eqref{bih}, $S$-functional calculus, see Definition \ref{Sfun} together with the assumption that $f$ is intrinsic (cf. \cite[Prop. 5.11]{CDP25} and \cite[Thm. 3.2.11]{CGK}), and \eqref{corol D-resolent} we get
		\begin{align*}
			&f_D(T)g(T)+f(\overline{T})g_D(T)\\
			&=\frac{1}{(2\pi)^2}\int_{\partial(G_2\cap\mathbb{C}_I)}\int_{\partial(G_1\cap\mathbb{C}_I)} f(s)ds_I [S_{D,R}^{-1} (s,T) S_{L}^{-1} (p,T)+S_{R}^{-1} (s,\overline{T}) S_{D,L}^{-1} (p,T)] dp_I g(p)\\
			&=\frac{1}{(2\pi)^2}\int_{\partial(G_2\cap\mathbb{C}_I)}\int_{\partial(G_1\cap\mathbb{C}_I)} f(s)ds_I [(S^{-1}_{D,L}(s,T)-S^{-1}_{D,L}(p,T))p -\bar{s}(S^{-1}_{D,L}(s,T)-S^{-1}_{D,L}(p,T))] Q_s^{-1}(p)dp_I g(p).
		\end{align*}
		Since the maps $p\mapsto pQ_s^{-1}(p)$ and $p\mapsto Q_s^{-1}(p)$ are intrinsic slice hyperholomorphic on $\overline{G}_1$, by \cite[Cor. 2.8.3]{ColomboSabadiniStruppa2011} the right-hand side of the above expression simplifies as
		\begin{align*}
			\frac{1}{(2\pi)^2}\int_{\partial(G_2\cap\mathbb{C}_I)}\int_{\partial(G_1\cap\mathbb{C}_I)} f(s)ds_I [\bar{s}S^{-1}_{D,L}(s,T) -S^{-1}_{D,L}(p,T)p] Q_s^{-1}(p) dp_I g(p). \end{align*}
		Finally, by \cite[Lemma 3.18]{ACGS} (with $B:=S^{-1}_{D,L}(s,T)$) we get
		\begin{align*}
			&\frac{1}{(2\pi)^2}\int_{\partial(G_2\cap\mathbb{C}_I)}\int_{\partial(G_1\cap\mathbb{C}_I)} f(s)ds_I [\bar{s}S^{-1}_{D,L}(s,T) -S^{-1}_{D,L}(p,T)p] Q_s^{-1}(p) dp_I g(p)\\
			&= \frac{1}{(2\pi)^2}\int_{\partial(G_1\cap\mathbb{C}_I)}  S_{D,L}^{-1} (p,T) dp_I f(p)g(p)\\
			& =(fg)_D(T).\qedhere
		\end{align*}
	\end{proof}
	
	\begin{remark}
		We observe that in \eqref{preres} we can write
		$$
		\mathcal{Q}_{c,s}^{-1}(T) S^{-1}_L(p,T) - 2 \mathcal{Q}_{c,s}^{-1}(T) \,\underline{T}\, \mathcal{Q}_{c,p}^{-1}(T) = \mathcal{Q}_{c,s}^{-1}(T) S^{-1}_L(p, \overline{T}).
		$$
		From this, we can deduce the following equivalent product rule for the biharmonic functional calculus:
		$$(fg)_D(T)=f_D(T)g(\overline{T})+f(T)g_D(T).$$
	\end{remark}
	
	Another possible application of the resolvent equation for the biharmonic functional calculus is given by the following result.
	
	\begin{theorem}\label{D-Riesz projectors}
		Let $T\in \mathcal{BC}^{0,1}(V_5)$ be such that its components $T_i$, $i=1,2,3,4,5$ have real spectra and $T_0=0$.
		Assume $\sigma_S(T)=\sigma_{S,1}(T)\cup \sigma_{S,2}(T)$ with $\text{dist}(\sigma_{S,1}(T), \sigma_{S,2}(T))>0.$ Let $G_1,G_2\subset\mathbb{R}^{n+1}$ be two admissible sets for $T$ such that $\sigma_{S,1}(T)\subset G_1$ and $\overline{G_1}\subset G_2$ such that dist$( G_2, \sigma_{S,2}(T))>0$. Then the following operator is a projector
		\begin{equation}
			P := \frac{1}{32\pi} \int_{\partial (G_1\cap\mathbb{C}_I)}  S^{-1}_{D,R}(p,T) dp_I p = \frac{1}{32\pi} \int_{\partial (G_2\cap\mathbb{C}_I)}   s ds_I   S^{-1}_{D,L}(s,T).
		\end{equation}
		
	\end{theorem}
	\begin{proof}
		We multiply equation \eqref{preres} on the right by $p$, integrate with respect to  $ds_I$ over $\partial(G_2\cap\mathbb{C}_I)$ and with respect to $dp_I$ over $\partial(G_1\cap\mathbb{C}_I)$ to obtain the following expression
		\begin{align}\label{int}
			&\int_{\partial (G_1\cap\mathbb{C}_I)}\int_{\partial (G_2\cap\mathbb{C}_I)}    ds_I \Big(S^{-1}_R(s,T)\mathcal{Q}_{c,p}^{-1}(T)+\mathcal{Q}_{c,s}^{-1}(T)S_{L}^{-1}(p,T) - 2\mathcal{Q}_{c,s}^{-1}(T)\underline{T}\mathcal{Q}_{c,p}^{-1}(T) \Big) dp_I p\\
			\nonumber
			&=\int_{\partial (G_1\cap\mathbb{C}_I)}\int_{\partial (G_2\cap\mathbb{C}_I)}   ds_I \Big[(S^{-1}_{D,L}(s,T)-S^{-1}_{D,L}(p,T))p -\bar{s}((S^{-1}_{D,L}(s,T)-S^{-1}_{D,L}(p,T))\Big] Q_s^{-1}(p)dp_I p.
		\end{align}
		By \eqref{eq:1.7} we have $S^{-1}_R(s,T)=s\mathcal{Q}_{c,s}^{-1}(T)-\mathcal{Q}_{c,s}^{-1}(T)\overline{T}$. Thus by \cite[Thm. 5.12]{CDP25} we can write the above equation in the following way:
		\begin{align}
			\nonumber
			& \int_{\partial (G_1\cap\mathbb{C}_I)}\int_{\partial (G_2\cap\mathbb{C}_I)}    ds_I  \Big([s\mathcal{Q}_{c,s}^{-1}(T)-\mathcal{Q}_{c,s}^{-1}(T)\overline{T}]\mathcal{Q}_{c,p}^{-1}(T)
			+\mathcal{Q}_{c,s}^{-1}(T)S_{L}^{-1}(p,T) - 2\mathcal{Q}_{c,s}^{-1}(T)\underline{T}\mathcal{Q}_{c,p}^{-1}(T) \Big) dp_I p\\
			\label{p0}
			&= \int_{\partial (G_1\cap\mathbb{C}_I)} ds_I  s\mathcal{Q}_{c,s}^{-1}(T) \int_{\partial (G_2\cap\mathbb{C}_I)}  \mathcal{Q}_{c,p}^{-1}(T)dp_I p=64 \pi^2 P^2.
		\end{align}
		Now, we focus on the right-hand-side of \eqref{int}. Due to slice the hyperholomorphicity of the functions $p\mapsto Q_s^{-1}(p)$, $p\mapsto pQ_s^{-1}(p)$ and the Cauchy integral theorem, see \cite[Cor. 2.8.3]{ColomboSabadiniStruppa2011}, and \cite[Lemma 3.18]{ACGS}, we have
		\begin{align}
			\nonumber
			&\int_{\partial (G_1\cap\mathbb{C}_I)}\int_{\partial (G_2\cap\mathbb{C}_I)}    ds_I [(S^{-1}_{D,L}(s,T)-(S^{-1}_{D,L}(p,T))p -\bar{s}((S^{-1}_{D,L}(s,T)-(S^{-1}_{D,L}(p,T))] Q_s^{-1}(p)dp_I p\\
			\nonumber
			&= \int_{\partial (G_1\cap\mathbb{C}_I)}\int_{\partial (G_2\cap\mathbb{C}_I)}    ds_I [\bar{s}S^{-1}_{D,L}(p,T)-S^{-1}_{D,L}(p,T)p] Q_s^{-1}(p)dp_I p\\
			&= 2\pi\int_{\partial (G_1\cap\mathbb{C}_I)}S^{-1}_{D,L}(p,T) dp_I p=64 \pi^2 P.
			\label{p11}
		\end{align}
		By putting plugging \eqref{p11} and \eqref{p0} into \eqref{int} we get the result.
	\end{proof}

	\section{The harmonic case associate with the $\Delta D$ operator}
	
	In this section, we establish a resolvent equation for the harmonic functional calculus based on the $S$-spectrum, arising from the factorization in \eqref{12}. We then derive the corresponding product rule and Riesz projectors. The technique is more involved than the one used to obtain the resolvent equation in the previous section.
	
	The harmonic functional calculus in dimension $n=5$ was first established in \cite{Fivedim}, where the authors applied the operator $\Delta_6 D$ to the slice hyperholomorphic Cauchy kernel and obtained
	$$
	D \Delta_6 S^{-1}_L(s,x) = D \Delta_6 S^{-1}_R(s,x) = 16 \mathcal{Q}_{c,s}^{-2}(x).
	$$
	By combining this result with the Cauchy formula (see \eqref{integrals}), it is possible to obtain an integral representation for functions in the set $\mathcal{APH}_1(U)$ (see \eqref{A1}). This leads to the following notion.

	\begin{definition}[Harmonic functional calculus based on the $S$-spectrum]
		Let $T\in \mathcal{BC}^{0,1}(V_5)$ be such that its components $T_i$, $i=1,2,3,4,5$ have real spectra and $T_0=0$ (see Remark \ref{FCOMPON}). Let $U$ be a bounded slice Cauchy domain, and set $ds_I = ds(-I)$ for $I \in \mathbb{S}$.
		For any function $f \in \mathcal{SH}_{\sigma_S(T)}^L(U)$ (resp.\ $f \in \mathcal{SH}_{\sigma_S(T)}^R(U)$), let $ f_{D \Delta}=D \Delta_6 f(x)$ (resp. $ f_{D \Delta}= f(x)D \Delta_6$ ). We define the harmonic functional calculus based on the $S$-spectrum for $T$ as
		\begin{equation}
			\label{har}
			f_{D \Delta}(T)= \frac{1}{2 \pi} \int_{\partial(U \cap \mathbb{C}_I)} S_{D \Delta,L}^{-1}(s,T)ds_I f(s), \qquad \left(\hbox{resp.} \, \, f_{D \Delta}(T)= \frac{1}{2 \pi} \int_{\partial(U \cap \mathbb{C}_I)} f(s)ds_IS_{D \Delta, R}^{-1}(s,T) \right)
		\end{equation}	
		where the resolvent operators $S_{\Delta D,L}^{-1}(s,T)$ and $S_{\Delta D,R}^{-1}(s,T)$ are defined as
		\begin{equation}
			\label{H2}
			S_{\Delta D, L}^{-1}(s,T) = S_{\Delta D,R}^{-1}(s,T)=16 \mathcal{Q}_{c,s}^{-2}(T).
		\end{equation}
	\end{definition}
	
	In order to derive a resolvent equation for the harmonic functional calculus, we first recall the notion of the left (resp. right) $F$-resolvent operators, which are defined as follows:
	\begin{align}\label{F-resol.op}
		F_L(s,T)= 64S_L^{-1}(s,T)\mathcal{Q}^{-2}_{c,s}(T) ,\qquad \left( \hbox{resp.} \, F_R(s,T)=64\mathcal{Q}^{-2}_{c,s}(T)S_R^{-1}(s,T) \right).
	\end{align}
	For  $T \in \mathcal{BC}^{0,1}(V_5)$ and $s \in \rho_S(T)$, The $F$-resolvent operators satisfy the equations (see \cite[Thm. 7.3.1]{CGK}):
	\begin{equation}
		\label{rf1}
		F_L(s,T)s-TF_L(s,T)=64\mathcal{Q}_{c,s}^{-2}(T),
	\end{equation}
	\begin{equation}
		\label{rf2}
		sF_R(s,T)-F_R(s,T)T=64 \mathcal{Q}_{c,s}^{-2}(T).
	\end{equation}
	
	\begin{remark}
		The operators defined in \eqref{F-resol.op} are crucial to establish the $F$-functional calculus; see \cite{FABJONA}. This is a monogenic functional calculus equivalent to the one established by McIntosh (see \cite{JM}); see \cite{CDP25}.
	\end{remark}
	
	Before getting the resolvent equation for the harmonic-functional calculus, we require some technical lemmas.
	\begin{lemma}
		Let $T\in \mathcal{BC}^{0,1}(V_5)$. For $p,s\in\rho_S(T)$ with $s\not\in[p]$, we have the following equation:
		\begin{align}
			\begin{split}
				\label{c3}	
				&S_R^{-1}(s,T) \mathcal{Q}_{c,p}^{-2}(T)+\mathcal{Q}_{c,s}^{-2}(T)S_L^{-1}(p,T)=[(\mathcal{Q}_{c,s}^{-2}(T)-\mathcal{Q}_{c,p}^{-2}(T))p-\bar{s}(\mathcal{Q}_{c,s}^{-2}(T)-\mathcal{Q}_{c,p}^{-2}(T))] Q_s^{-1}(p)\\
				& +[(\mathcal{Q}_{c,s}^{-2}(T)T S_{L}^{-1}(p,T)-s \mathcal{Q}_{c,s}^{-2}(T) S_{L}^{-1}(p,T)+S_R^{-1}(s,T) \mathcal{Q}_{c,p}^{-2}(T)p-S_R^{-1}(s,T) T \mathcal{Q}_{c,p}^{-2}(T))p\\
				&-\bar{s}(\mathcal{Q}_{c,s}^{-2}(T)T S_{L}^{-1}(p,T)-s \mathcal{Q}_{c,s}^{-2}(T) S_{L}^{-1}(p,T)+S_R^{-1}(s,T) \mathcal{Q}_{c,p}^{-2}(T)p-S_R^{-1}(s,T) T \mathcal{Q}_{c,p}^{-2}(T))p] Q_s^{-1}(p),
			\end{split}
		\end{align}
		where $Q_s(p):=p^2-2s_0p+|s|^2$.
	\end{lemma}
	\begin{proof}	
		We first consider $B = T$ in \eqref{Breseq} and multiply the equation on the right by $-64\mathcal{Q}^{-2}_{c,p}(T)$. By Lemma \ref{Lemmastep0} we obtain
		\begin{align}\label{s1}
			\begin{split}
				&-S^{-1}_R(s,T)TF_L(p,T)=64 [(-S_R^{-1}(s,T)T\mathcal{Q}^{-2}_{c,p}(T)+TS_L^{-1}(p,T)\mathcal{Q}^{-2}_{c,p}(T))p\\
				&-\overline{s}(-S_R^{-1}(s,T)T\mathcal{Q}^{-2}_{c,p}(T)
				+TS_L^{-1}(p,T)\mathcal{Q}^{-2}_{c,p}(T))]Q_s^{-1}(p).
			\end{split}
		\end{align}
		We consider the case $B = \mathcal{I}$ in \eqref{Bres} and multiply the equation on the right by $64\mathcal{Q}^{-2}_{c,p}(T)p$. By Lemma \ref{Lemmastep0}, we then obtain
		\begin{align}\label{s2}
			\begin{split}
				&S^{-1}_R(s,T)F_L(p,T)p=64 [(S_R^{-1}(s,T)\mathcal{Q}^{-2}_{c,p}(T)p-S_L^{-1}(p,T)\mathcal{Q}^{-2}_{c,p}(T)p)p\\
				&-\overline{s}(S_R^{-1}(s,T)\mathcal{Q}^{-2}_{c,p}(T)p
				-S_L^{-1}(p,T)\mathcal{Q}^{-2}_{c,p}(T)p)]Q_s^{-1}(p).
			\end{split}
		\end{align}
		We sum \eqref{s1} and \eqref{s2}, and by using \eqref{rf2} we get
		\begin{align}
			\nonumber
			&S^{-1}_R(s,T)\mathcal{Q}^{-2}_{c,p}(T)=[(S_R^{-1}(s,T)\mathcal{Q}^{-2}_{c,p}(T)p-S_R^{-1}(s,T)T \mathcal{Q}_{c,p}^{-2}(T)- \mathcal{Q}_{c,p}^{-2}(T))p\\
			\label{c1}
			&-\overline{s}(S_R^{-1}(s,T)\mathcal{Q}^{-2}_{c,p}(T)p
			-S_R^{-1}(s,T)T\mathcal{Q}_{c,p}^{-2}(T)-\mathcal{Q}_{c,s}^{-2}(T))]Q_s^{-1}(p).
		\end{align}
		We now perform analogous computations on the left-hand side. More precisely, we first consider $B = T$ in \eqref{Breseq} and multiply the resulting equation on the left by $-64\mathcal{Q}^{-2}_{c,s}(T)$. Next, we take $B = \mathcal{I}$ and multiply the equation on the left by $64 s \mathcal{Q}^{-2}_{c,s}(T)$. Summing the two resulting expressions and using \eqref{rf1}, we obtain
		\begin{align}
			\nonumber
			& \mathcal{Q}_{c,s}^{-2}(T)S_L^{-1}(p,T) =  [(\mathcal{Q}_{c,s}^{-2}(T)+\mathcal{Q}_{c,s}^{-2}(T)TS_L^{-1}(p,T)-s \mathcal{Q}_{c,s}^{-2}(T)S_L^{-1}(p,T))p\\
			\label{c2}
			&-\bar{s} (\mathcal{Q}_{c,s}^{-2}(T)+\mathcal{Q}_{c,s}^{-2}(T)TS_L^{-1}(p,T)-s \mathcal{Q}_{c,s}^{-2}(T)S_L^{-1}(p,T))] Q_s^{-1}(p).
		\end{align}
		
		Summing \eqref{c1} and \eqref{c2} we obtain \eqref{c3}.
	\end{proof}
	\begin{lemma}
		Let $T\in \mathcal{BC}^{0,1}(V_5)$.	For $p$, $s\in\rho_S(T)$ with $s\not\in[p]$, we have the following equation:
		\begin{align}
			\label{c10}
			&S_R^{-1}(s,T) \mathcal{Q}_{c,p}^{-2}(T)+\mathcal{Q}_{c,s}^{-2}(T)S_L^{-1}(p,T)-2 \mathcal{Q}_{c,s}^{-1}(T) S_R^{-1}(s,T) T S_L^{-1}(p,T) \mathcal{Q}_{c,p}^{-1}(T)\\
			\nonumber
			& +\mathcal{Q}_{c,s}^{-1}(T)S^{-1}_R(s,T)S_{L}^{-1}(p,T) \mathcal{Q}_{c,p}^{-1}(T)p+s\mathcal{Q}_{c,s}^{-1}(T)S^{-1}_R(s,T)S_{L}^{-1}(p,T) \mathcal{Q}_{c,p}^{-1}(T)\\
			\nonumber
			&-2\mathcal{Q}_{c,s}^{-2}(T)\underline{T} \mathcal{Q}_{c,p}^{-1}(T)-2 \mathcal{Q}_{c,s}^{-1}(T)\underline{T} \mathcal{Q}_{c,p}^{-2 }(T)\\
			\nonumber
			&=[(\mathcal{Q}_{c,s}^{-2}(T)-\mathcal{Q}_{c,p}^{-2}(T))p-\bar{s}(\mathcal{Q}_{c,s}^{-2}(T)-\mathcal{Q}_{c,p}^{-2}(T))] Q_s^{-1}(p).
		\end{align}
		where $\underline{T}= T_1 e_1+T_2 e_2+T_3 e_3+T_4 e_4+T_5 e_5$.
	\end{lemma}
	\begin{proof} Let $B=-2T$ in equation \eqref{Breseq}, multiply on the left by $ \mathcal{Q}_{c,s}^{-1}(T)$ and on the right by $\mathcal{Q}_{c,p}^{-1}(T)$. By Lemma \ref{Lemmastep0} we get
		\begin{eqnarray}
			\nonumber
			&&- \mathcal{Q}_{c,s}^{-1}(T) S_R^{-1}(s,T) 2T S_L^{-1}(p,T) \mathcal{Q}_{c,p}^{-1}(T)\\
			\nonumber
			&=&[(-\mathcal{Q}_{c,s}^{-1}(T)S_{R}^{-1}(s,T)2T \mathcal{Q}_{c,p}^{-1}(T)+ \mathcal{Q}_{c,s}^{-1}(T)2T S_{L}^{-1}(p,T) \mathcal{Q}_{c,p}^{-1}(T))p\\
			\label{c4}
			&&-\bar{s}(-\mathcal{Q}_{c,s}^{-1}(T)S_{R}^{-1}(s,T)2T \mathcal{Q}_{c,p}^{-1}(T)+ \mathcal{Q}_{c,s}^{-1}(T)2T S_{L}^{-1}(p,T) \mathcal{Q}_{c,p}^{-1}(T))] Q_s^{-1}(p).
		\end{eqnarray}
		Let $B = \mathcal{I}$ in equation \eqref{Breseq}, multiply it on the left by $\mathcal{Q}_{c,s}^{-1}(T)$ and on the right by $\mathcal{Q}_{c,p}^{-1}(T)p$.
		Next, take $B = \mathcal{I}$ in \eqref{Breseq}, multiply it on the left by $s \, \mathcal{Q}_{c,s}^{-1}(T)$ and on the right by $\mathcal{Q}_{c,p}^{-1}(T)$, and use Lemma \ref{Lemmastep0}.
		Adding the expression gives
		\begin{equation}
			\label{N1}
			\mathcal{Q}_{c,s}^{-1}(T)S^{-1}_R(s,T)S_{L}^{-1}(p,T) \mathcal{Q}_{c,p}^{-1}(T)p+s\mathcal{Q}_{c,s}^{-1}(T)S^{-1}_R(s,T)S_{L}^{-1}(p,T) \mathcal{Q}_{c,p}^{-1}(T)= \mathcal{B}(s,p,T),
		\end{equation}
		where
		\begin{eqnarray*}
			\mathcal{B}(s,p,T)&:=&[(\mathcal{Q}_{c,s}^{-1}(T)S^{-1}_R(s,T)\mathcal{Q}_{c,p}^{-1}(T)p- \mathcal{Q}_{c,s}^{-1}(T)S_L^{-1}(p,T) \mathcal{Q}_{c,p}^{-1}(T)p\\
			&&+s\mathcal{Q}_{c,s}^{-1}(T)S^{-1}_R(s,T)\mathcal{Q}_{c,p}^{-1}(T)- s\mathcal{Q}_{c,s}^{-1}(T)S_L^{-1}(p,T) \mathcal{Q}_{c,p}^{-1}(T))p\\
			&&- \bar{s}(\mathcal{Q}_{c,s}^{-1}(T)S^{-1}_R(s,T)\mathcal{Q}_{c,p}^{-1}(T)p- \mathcal{Q}_{c,s}^{-1}(T)S_L^{-1}(p,T) \mathcal{Q}_{c,p}^{-1}(T)p\\
			&&+s\mathcal{Q}_{c,s}^{-1}(T)S^{-1}_R(s,T)\mathcal{Q}_{c,p}^{-1}(T)- s\mathcal{Q}_{c,s}^{-1}(T)S_L^{-1}(p,T) \mathcal{Q}_{c,p}^{-1}(T))] Q_s^{-1}(p).
		\end{eqnarray*}

		We sum \eqref{c3}, \eqref{c4} and \eqref{N1} and we get
		\begin{align}
			\nonumber
			&S_R^{-1}(s,T) \mathcal{Q}_{c,p}^{-2}(T)+\mathcal{Q}_{c,s}^{-2}(T)S_L^{-1}(p,T)- 2\mathcal{Q}_{c,s}^{-1}(T) S_R^{-1}(s,T) T S_L^{-1}(p,T) \mathcal{Q}_{c,p}^{-1}(T)\\
			\nonumber
			&+\mathcal{Q}_{c,s}^{-1}(T)S^{-1}_R(s,T)S_{L}^{-1}(p,T) \mathcal{Q}_{c,p}^{-1}(T)p+s\mathcal{Q}_{c,s}^{-1}(T)S^{-1}_R(s,T)S_{L}^{-1}(p,T) \mathcal{Q}_{c,p}^{-1}(T)\\
			\label{N2}
			&=[(\mathcal{Q}_{c,s}^{-2}(T)-\mathcal{Q}_{c,p}^{-2}(T))p-\bar{s}(\mathcal{Q}_{c,s}^{-2}(T)-\mathcal{Q}_{c,p}^{-2}(T))] Q_s^{-1}(p)+ \mathcal{A}_1(s,p,T)+ \mathcal{A}_2(s,p,T)+ \mathcal{B}(s,p,T),
		\end{align}
		where
		\begin{align*}
			\mathcal{A}_1(s,p,T)&=[(\mathcal{Q}_{c,s}^{-2}(T)T S_{L}^{-1}(p,T)-S_R^{-1}(s,T) T \mathcal{Q}_{c,p}^{-2}(T)\\
			&-\mathcal{Q}_{c,s}^{-1}(T)S_{R}^{-1}(s,T)T \mathcal{Q}_{c,p}^{-1}(T)+ \mathcal{Q}_{c,s}^{-1}(T)T S_{L}^{-1}(p,T) \mathcal{Q}_{c,p}^{-1}(T))p\\
			&-\bar{s}(\mathcal{Q}_{c,s}^{-2}(T)T S_{L}^{-1}(p,T)-S_R^{-1}(s,T) T \mathcal{Q}_{c,p}^{-2}(T)\\
			&-\mathcal{Q}_{c,s}^{-1}(T)S_{R}^{-1}(s,T)T \mathcal{Q}_{c,p}^{-1}(T)+ \mathcal{Q}_{c,s}^{-1}(T)T S_{L}^{-1}(p,T) \mathcal{Q}_{c,p}^{-1}(T)]Q_s^{-1}(p)
		\end{align*}
		and
		\begin{align*}
			\mathcal{A}_2(s,p,T)&=[(-s \mathcal{Q}_{c,s}^{-2}(T) S_{L}^{-1}(p,T)+S_R^{-1}(s,T) \mathcal{Q}_{c,p}^{-2}(T)p\\
			&-\mathcal{Q}_{c,s}^{-1}(T)S_{R}^{-1}(s,T)T \mathcal{Q}_{c,p}^{-1}(T)+ \mathcal{Q}_{c,s}^{-1}(T)T S_{L}^{-1}(p,T) \mathcal{Q}_{c,p}^{-1}(T))p\\
			&-\bar{s}[(-s \mathcal{Q}_{c,s}^{-2}(T) S_{L}^{-1}(p,T)+S_R^{-1}(s,T) \mathcal{Q}_{c,p}^{-2}(T)p\\
			&-\mathcal{Q}_{c,s}^{-1}(T)S_{R}^{-1}(s,T)T \mathcal{Q}_{c,p}^{-1}(T)+ \mathcal{Q}_{c,s}^{-1}(T)T S_{L}^{-1}(p,T) \mathcal{Q}_{c,p}^{-1}(T)]Q_s^{-1}(p).
		\end{align*}
		Now we simplify $\mathcal{A}_1(s,p,T)$, $\mathcal{A}_2(s,p,T)$ and $ \mathcal{B}(s,p,T)$. The definition of the $S$-resolvent operator implies
		\begin{eqnarray*}
			&&\mathcal{Q}_{c,s}^{-2}(T)T S_{L}^{-1}(p,T)-S_R^{-1}(s,T) T \mathcal{Q}_{c,p}^{-2}(T)-\mathcal{Q}_{c,s}^{-1}(T)S_{R}^{-1}(s,T)T \mathcal{Q}_{c,p}^{-1}(T)+ \mathcal{Q}_{c,s}^{-1}(T)T S_{L}^{-1}(p,T) \mathcal{Q}_{c,p}^{-1}(T)\\
			&=& \mathcal{Q}_{c,s}^{-2}(T)T(p\mathcal{I}-\overline{T}) \mathcal{Q}_{c,p}^{-1}(T)-\mathcal{Q}_{c,s}^{-1}(T) (s\mathcal{I}-\overline{T})T \mathcal{Q}_{c,p}^{-2}(T)-\mathcal{Q}_{c,s}^{-2}(T)(s\mathcal{I}-\overline{T})T \mathcal{Q}_{c,p}^{-1}(T)\\
			&&+\mathcal{Q}_{c,s}^{-1}(T) T(p\mathcal{I}-\overline{T}) \mathcal{Q}_{c,p}^{-2}(T)\\
			&=& \mathcal{Q}_{c,s}^{-2}(T)Tp \mathcal{Q}_{c,p}^{-1}(T)-\mathcal{Q}_{c,s}^{-1}(T)sT \mathcal{Q}_{c,p}^{-2}(T)- \mathcal{Q}_{c,s}^{-2}(T)sT \mathcal{Q}_{c,p}^{-1}(T)+ \mathcal{Q}_{c,s}^{-1}(T)Tp \mathcal{Q}_{c,p}^{-2}(T).
		\end{eqnarray*}
		Therefore, we can write
		\begin{align}
			\nonumber
			\mathcal{A}_1&(s,p,T)=[(\mathcal{Q}_{c,s}^{-2}(T)Tp \mathcal{Q}_{c,p}^{-1}(T)-\mathcal{Q}_{c,s}^{-1}(T)sT \mathcal{Q}_{c,p}^{-2}(T)- \mathcal{Q}_{c,s}^{-2}(T)sT \mathcal{Q}_{c,p}^{-1}(T)+ \mathcal{Q}_{c,s}^{-1}(T)Tp \mathcal{Q}_{c,p}^{-2}(T))p\\
			\nonumber
			&-\bar{s}(\mathcal{Q}_{c,s}^{-2}(T)Tp \mathcal{Q}_{c,p}^{-1}(T)-\mathcal{Q}_{c,s}^{-1}(T)sT \mathcal{Q}_{c,p}^{-2}(T)- \mathcal{Q}_{c,s}^{-2}(T)sT \mathcal{Q}_{c,p}^{-1}(T)+ \mathcal{Q}_{c,s}^{-1}(T)Tp \mathcal{Q}_{c,p}^{-2}(T))]Q_s^{-1}(p)\\
			\nonumber
			&=[\mathcal{Q}_{c,s}^{-2}(T)Tp^2 \mathcal{Q}_{c,p}^{-1}(T)-\mathcal{Q}_{c,s}^{-1}(T)sTp \mathcal{Q}_{c,p}^{-2}(T)- \mathcal{Q}_{c,s}^{-2}(T)sTp \mathcal{Q}_{c,p}^{-1}(T)+ \mathcal{Q}_{c,s}^{-1}(T)Tp^2 \mathcal{Q}_{c,p}^{-2}(T)+\\
			\nonumber
			&-\mathcal{Q}_{c,s}^{-2}(T)\bar{s}Tp \mathcal{Q}_{c,p}^{-1}(T)+\mathcal{Q}_{c,s}^{-1}(T)|s|^2T \mathcal{Q}_{c,p}^{-2}(T)+ \mathcal{Q}_{c,s}^{-2}(T)|s|^2 T \mathcal{Q}_{c,p}^{-1}(T)- \mathcal{Q}_{c,s}^{-1}(T)\bar{s}Tp \mathcal{Q}_{c,p}^{-2}(T)]Q_s^{-1}(p)\\
			\nonumber
			&=[ \mathcal{Q}_{c,s}^{-2}(T)T (p^2-2s_0p+|s|^2) \mathcal{Q}_{c,p}^{-1}(T)+ \mathcal{Q}_{c,s}^{-1}(T)T (p^2-2s_0p+|s|^2) \mathcal{Q}_{c,p}^{-2}(T)] Q_s^{-1}(p)\\
			\label{c8}
			&= \mathcal{Q}_{c,s}^{-2}(T)T \mathcal{Q}_{c,p}^{-1}(T)+ \mathcal{Q}_{c,s}^{-1}(T)T \mathcal{Q}_{c,p}^{-2 }(T).
		\end{align}

		We now focus on the right-hand side of $\mathcal{A}_2(s,p,T)$ and $\mathcal{B}(s,p,T)$.
		Thus, we have
		\begin{align*}
			&S^{-1}_R(s,T)p \mathcal{Q}_{c,p}^{-2}(T)-s \mathcal{Q}_{c,s}^{-2}(T) S^{-1}_L(p,T)+\mathcal{Q}_{c,s}^{-1}(T)S^{-1}_R(s,T)\mathcal{Q}_{c,p}^{-1}(T)p\\
			&- \mathcal{Q}_{c,s}^{-1}(T)S_L^{-1}(p,T) \mathcal{Q}_{c,p}^{-1}(T)p+s\mathcal{Q}_{c,s}^{-1}(T)S^{-1}_R(s,T)\mathcal{Q}_{c,p}^{-1}(T)- s\mathcal{Q}_{c,s}^{-1}(T)S_L^{-1}(p,T) \mathcal{Q}_{c,p}^{-1}(T)\\
			&- \mathcal{Q}_{c,s}^{-1}(T)S^{-1}_R(s,T)T \mathcal{Q}_{c,p}^{-1}(T)+\mathcal{Q}_{c,s}^{-1}(T)T S^{-1}_L(p,T) \mathcal{Q}_{c,p}^{-1}(T)\\
			=& S^{-1}_R(s,T)p \mathcal{Q}_{c,p}^{-2}(T)-s \mathcal{Q}_{c,s}^{-2}(T) S^{-1}_L(p,T)+\mathcal{Q}_{c,s}^{-1}(T)S^{-1}_R(s,T)\mathcal{Q}_{c,p}^{-1}(T)p\\
			&- s\mathcal{Q}_{c,s}^{-1}(T)S_L^{-1}(p,T) \mathcal{Q}_{c,p}^{-1}(T) +\mathcal{Q}_{c,s}^{-1}(T)[s S^{-1}_R(s,T)-S_{R}^{-1}(s,T)T] \mathcal{Q}_{c,p}^{-1}(T)\\
			&-\mathcal{Q}_{c,s}^{-1}(T)[ S^{-1}_L(p,T)p-TS_{L}^{-1}(p,T)] \mathcal{Q}_{c,p}^{-1}(T).
		\end{align*}
		By using \eqref{sl} and \eqref{sr}, together with the definition of the $S$-resolvent operators, we obtain
		\begin{align*}
			&S^{-1}_R(s,T)p \mathcal{Q}_{c,p}^{-2}(T)-s \mathcal{Q}_{c,s}^{-2}(T) S^{-1}_L(p,T)+\mathcal{Q}_{c,s}^{-1}(T)S^{-1}_R(s,T)\mathcal{Q}_{c,p}^{-1}(T)p- s\mathcal{Q}_{c,s}^{-1}(T)S_L^{-1}(p,T) \mathcal{Q}_{c,p}^{-1}(T)\\
			=& \mathcal{Q}_{c,s}^{-1}(T)(s\mathcal{I}-\overline{T})p \mathcal{Q}_{c,p}^{-2}(T)- \mathcal{Q}_{c,s}^{-2}(T)s (p\mathcal{I}-\overline{T}) \mathcal{Q}_{c,p}^{-1}(T)\\
			&+\mathcal{Q}_{c,s}^{-2}(T)(s\mathcal{I}-\overline{T})p \mathcal{Q}_{c,p}^{-1}(T)- \mathcal{Q}_{c,s}^{-1}(T)s (p\mathcal{I}-\overline{T}) \mathcal{Q}_{c,p}^{-2}(T)\\
			=&-\mathcal{Q}_{c,s}^{-1}(T) \overline{T}p \mathcal{Q}_{c,p}^{-2}(T)+ \mathcal{Q}_{c,s}^{-2}(T)s \overline{T} \mathcal{Q}_{c,p}^{-1}(T)-\mathcal{Q}_{c,s}^{-2}(T) \overline{T}p \mathcal{Q}_{c,p}^{-1}(T)+ \mathcal{Q}_{c,s}^{-1}(T)s \overline{T} \mathcal{Q}_{c,p}^{-2}(T).
		\end{align*}
		Therefore we have
		\begin{align}
			\mathcal{A}_2(s,p,T)&+\mathcal{B}(s,p,T)=\nonumber[(-\mathcal{Q}_{c,s}^{-1}(T) \overline{T}p \mathcal{Q}_{c,p}^{-2}(T)+ \mathcal{Q}_{c,s}^{-2}(T)s \overline{T} \mathcal{Q}_{c,p}^{-1}(T)\\
			\nonumber
			&-\mathcal{Q}_{c,s}^{-2}(T) \overline{T}p \mathcal{Q}_{c,p}^{-1}(T)+ \mathcal{Q}_{c,s}^{-1}(T)s \overline{T} \mathcal{Q}_{c,p}^{-2}(T))p\\
			\nonumber
			&-\bar{s}(-\mathcal{Q}_{c,s}^{-1}(T) \overline{T}p \mathcal{Q}_{c,p}^{-2}(T)+ \mathcal{Q}_{c,s}^{-2}(T)s \overline{T} \mathcal{Q}_{c,p}^{-1}(T)\\
			\nonumber
			&-\mathcal{Q}_{c,s}^{-2}(T) \overline{T}p \mathcal{Q}_{c,p}^{-1}(T)+ \mathcal{Q}_{c,s}^{-1}(T)s \overline{T} \mathcal{Q}_{c,p}^{-2}(T))] Q_s^{-1}(p)\\
			\nonumber
			&= [\mathcal{Q}_{c,s}^{-1}(T) \overline{T} (-p^2+2s_0p-|s|^2) \mathcal{Q}_{c,p}^{-2}(T)+\mathcal{Q}_{c,s}^{-2}(T) \overline{T} (-p^2+2s_0p-|s|^2) \mathcal{Q}_{c,p}^{-1}(T)] Q_s^{-1}(p)\\
			\label{c7}
			&=- [\mathcal{Q}_{c,s}^{-1}(T) \overline{T}  \mathcal{Q}_{c,p}^{-2}(T)+\mathcal{Q}_{c,s}^{-2}(T) \overline{T}  \mathcal{Q}_{c,p}^{-1}(T)].
		\end{align}
		Plugging \eqref{c8} and \eqref{c7} into \eqref{N2} and we get \eqref{c10}.
	\end{proof}

	\begin{theorem}[Resolvent equation for the harmonic functional calculus]\label{DDelta-resolventequation}
		Let $T \in \mathcal{BC}^{0,1}(V_5)$. For $p$, $s\in\rho_S(T)$ with $s\not\in[p]$, we have the following:
		\begin{align}\label{DDelta-resolvent eq.}
			\begin{split}
				&S_R^{-1}(s,\overline{T}) S_{D\Delta,L}^{-1}(p,T) + S_{D\Delta,R}^{-1}(s,T)  S_L^{-1}(p,\overline{T}) + \frac{1}{2}S_{D,R}^{-1}(s,T) S_{\Delta,L}^{-1}(p,T) +  \frac{1}{2}S_{\Delta, R}^{-1}(s,T)  S_{D, L}^{-1}(p,T)\\
				&= [(S^{-1}_{\Delta D,L}(s,T)-S^{-1}_{\Delta D,R}(p,T))p- \bar{s}(S^{-1}_{\Delta,L D}(s,T)-S^{-1}_{\Delta D,R}(p,T))]Q_s^{-1}(p),
			\end{split}
		\end{align}
		and
		\begin{align}\label{DDelta-resolvent eq1}
			\begin{split}
				&S_R^{-1}(s,T) S_{D\Delta,L}^{-1}(p,T) + S_{D\Delta,R}^{-1}(s,T) S_L^{-1}(p,T) + \frac{1}{2} S_{D,R}^{-1}(s,T) S_{\Delta,L}^{-1}(p,\overline{T}) + \frac{1}{2} S_{\Delta,R}^{-1}(s,\overline{T}) S_{D,L}^{-1}(p,T) \\
				&= \big[(S^{-1}_{\Delta D,L}(s,T) - S^{-1}_{\Delta D,R}(p,T))p - \bar{s} (S^{-1}_{\Delta D,L}(s,T) - S^{-1}_{\Delta D,R}(p,T)) \big] Q_s^{-1}(p).
			\end{split}
		\end{align}
	\end{theorem}
	
	\begin{proof}
		First we prove \eqref{DDelta-resolventequation}. By using \eqref{rK1} and \eqref{rK2} in the left-hand side of \eqref{c10}, we obtain
		\begin{align}
			\label{c11}
			&S_R^{-1}(s,T) \mathcal{Q}_{c,p}^{-2}(T)+\mathcal{Q}_{c,s}^{-2}(T)S_L^{-1}(p,T)- \mathcal{Q}_{c,s}^{-1}(T) S_R^{-1}(s,T) T S_L^{-1}(p,T) \mathcal{Q}_{c,p}^{-1}(T)\\
			\nonumber
			&- \mathcal{Q}_{c,s}^{-1}(T) S_R^{-1}(s,T) T S_L^{-1}(p,T) \mathcal{Q}_{c,p}^{-1}(T) +\mathcal{Q}_{c,s}^{-1}(T)S^{-1}_R(s,T)S_{L}^{-1}(p,T) \mathcal{Q}_{c,p}^{-1}(T)p\\
			\nonumber
			&+s\mathcal{Q}_{c,s}^{-1}(T)S^{-1}_R(s,T)S_{L}^{-1}(p,T) \mathcal{Q}_{c,p}^{-1}(T)- 2\mathcal{Q}_{c,s}^{-1}(T) \underline{T} \mathcal{Q}_{c,p}^{-2}(T)- 2\mathcal{Q}_{c,s}^{-2}(T) \underline{T} \mathcal{Q}_{c,p}^{-1}(T)\\
			\nonumber
			=& S_R^{-1}(s,T) \mathcal{Q}_{c,p}^{-2}(T)+\mathcal{Q}_{c,s}^{-2}(T)S_L^{-1}(p,T)+\frac{1}{64}[s S^{-1}_{ \Delta,R}(s,T)-S^{-1}_{\Delta,R}(s,T)T]S^{-1}_{\Delta,L}(p,T)\\
			\nonumber
			& +\frac{1}{64}S^{-1}_{ \Delta,R}(s,T)[S^{-1}_{\Delta,L}(p,T)p-TS^{-1}_{\Delta,L}(p,T)] - 2\mathcal{Q}_{c,s}^{-1}(T) \underline{T} \mathcal{Q}_{c,p}^{-2}(T)- 2\mathcal{Q}_{c,s}^{-2}(T) \underline{T} \mathcal{Q}_{c,p}^{-1}(T)\\
			\nonumber
			=&  	S^{-1}_R(s,T) \mathcal{Q}_{c,p}^{-2}(T)- 2\mathcal{Q}_{c,s}^{-1}(T) \underline{T} \mathcal{Q}_{c,p}^{-2}(T)+\mathcal{Q}_{c,s}^{-2}(T)S^{-1}_L(p,T)- 2\mathcal{Q}_{c,s}^{-2}(T) \underline{T} \mathcal{Q}_{c,p}^{-1}(T)\\
			\nonumber
			&-\frac{1}{8}\mathcal{Q}_{c,s}^{-1}(T) S^{-1}_{ \Delta,L}(p,T)-\frac{1}{8}S^{-1}_{ \Delta,R}(s,T) \mathcal{Q}_{c,p}^{-1}(T)\\
			\nonumber
			=&S^{-1}_R(s,\overline{T}) \mathcal{Q}_{c,p}^{-2}(T)+\mathcal{Q}_{c,s}^{-2}(T)S^{-1}_L(p,\overline{T})-\frac{1}{8}\mathcal{Q}_{c,s}^{-1}(T) S^{-1}_{ \Delta,L}(p,T)-\frac{1}{8}S^{-1}_{ \Delta,R}(s,T) \mathcal{Q}_{c,p}^{-1}(T).
		\end{align}
		Substituting \eqref{c11}, into \eqref{c10} and multiplying by $16$, we obtain \eqref{DDelta-resolvent eq.}. Finally, \eqref{DDelta-resolvent eq1} follows from the penultimate equalities in \eqref{c11}, \eqref{c10} and by observing that
		$$
		\frac{1}{8} \mathcal{Q}_{c,s}^{-1}(T) S^{-1}_{\Delta,L}(p,T) + 2 \mathcal{Q}_{c,s}^{-1}(T) \,\underline{T}\, \mathcal{Q}_{c,p}^{-2}(T) = \frac{1}{8} \mathcal{Q}_{c,s}^{-1}(T) S^{-1}_{\Delta,L}(p, \overline{T}),
		$$
		$$
		\frac{1}{8} S^{-1}_{\Delta,R}(s,T)\, \mathcal{Q}_{c,p}^{-1}(T) + 2 \mathcal{Q}_{c,s}^{-2}(T) \,\underline{T}\, \mathcal{Q}_{c,p}^{-1}(T) = \frac{1}{8} S^{-1}_{\Delta,R}(p, \overline{T}) \, \mathcal{Q}_{c,p}^{-1}(T).
		$$
	\end{proof}

	\begin{theorem}[Product rule for the harmonic functional calculus]\label{DDelta prod. rule}
		Let $T\in \mathcal{BC}^{0,1}(V_5)$ be such that its components $T_i$, $i=1,2,3,4,5$ have real spectra and $T_0=0$.
		Suppose that $f\in\mathcal{N}_{\sigma_S(T)}(U)$ and $g\in\mathcal{SH}^L_{\sigma_S(T)}(U)$. Then, we have
		\begin{align}
			\label{one}
			(fg)_{(D\Delta) }(T)&= f(\overline{T}) g_{D\Delta}(T) +  f_{D\Delta}(T) g(\overline{T}) + \frac{1}{2}f_D(T)g_\Delta(T)+\frac{1}{2}f_\Delta(T)g_D(T)\\
			\label{second}
			&=f(T) g_{D\Delta}(T) +  f_{D\Delta}(T) g(T) + \frac{1}{2}f_D(T)g_\Delta(\overline{T})+\frac{1}{2}f_\Delta(\overline{T})g_D(T).
		\end{align}
	\end{theorem}
	\begin{proof}
		The identity \eqref{one} follows from \eqref{DDelta-resolvent eq.} by applying arguments analogous to those used in Theorem~\ref{D prod. rule}. Similarly, the identity \eqref{second} follows from \eqref{DDelta-resolvent eq1}, again using arguments analogous to those employed in Theorem~\ref{D prod. rule}.
	\end{proof}
	Also for the harmonic functional calculus, it is possible to define counterparts of the Riesz projectors. The following theorem can be proved using arguments similar to those in Theorem~\ref{D-Riesz projectors}.

	\begin{theorem}\label{DDelta-Riesz projectors}
		Under the same assumptions as Theorem \ref{D-Riesz projectors}, the following operator is a projector
		\begin{equation}
			P := \frac{1}{8\pi} \int_{\partial (G_1\cap\mathbb{C}_I)}  S^{-1}_{D\Delta,R}(p,T) dp_I p^3 = \frac{1}{8\pi} \int_{\partial (G_2\cap\mathbb{C}_I)}   s^3 ds_I   S^{-1}_{D\Delta,L}(s,T).
		\end{equation}
	\end{theorem}
	
	{\bf Data availability}. The research in this paper does not imply use of data.

{\bf Conflict of interest}. The authors declare that there is no conflict of interest.

\end{document}